\documentclass[11pt]{amsart}
\usepackage{amsfonts}
\usepackage{graphicx}
\usepackage{amscd}

\usepackage{amsthm}
\usepackage{amstext}
\usepackage{amsmath}
\usepackage{amssymb}
\usepackage[mathscr]{eucal}
\usepackage{url}
\usepackage{verbatim}
\usepackage{mathtools}
\usepackage{booktabs}

\usepackage{tikz}
\newcommand*\circled[1]{\tikz[baseline=(char.base)]{
   \node[shape=circle,draw,inner sep=1pt] (char) {#1};}}

\usepackage[official]{eurosym}

\newtheorem{theorem}{Theorem}
\theoremstyle{plain}

\newtheorem{conjecture}{Conjecture}
\newtheorem{corollary}{Corollary}

\newtheorem{definition}{Definition}

\newtheorem{lemma}{Lemma}

\newtheorem{proposition}{Proposition}
\newtheorem{remark}{Remark}

\numberwithin{equation}{section}

\begin{document}

\title[Periods of a max-type equation]{Periods of a max-type equation } 
\author{A. Linero Bas and D. Nieves Rold\'an}
\begin{abstract}
We consider the max-type equation $$x_{n+4}=\max\{x_{n+3},x_{n+2},x_{n+1},0\}-x_n,$$ with arbitrary real initial conditions. We describe completely its set of periods $\mathrm{Per}(F_4)$, as well as its associate periodic orbits. We also prove that there exists a natural number $N\notin\mathrm{Per}(F_4)$  for which 
$$\left\{N+m:m\geq 1,m\in\mathbb N\right\}\subset\mathrm{Per}(F_4).$$ 
\end{abstract}
\maketitle
This is a draft. The final version will be published in Journal of Difference Equations and Applications.

\section{Introduction}

In the last decades the interest on the behaviour of the solutions of max-type difference equations has rapidly grown. As a first comment, let us mention that they can be related to some piecewise linear difference equations, see \cite{Lozi}, by means of suitable changes of variables. With this procedure, from a piecewise model, we find the well-known max-type generalized Lyness difference equation, $$x_{n+1} = \frac{\max\{ x_n^k, A\}}{x_n^{\ell}x_{n-1}}, $$
where $A$ is a positive real number and the exponents $k$ and $\ell$ are integer numbers (for details, consult \cite{Ladas95} and \cite{GroveL05}).

That equation motivated a lot of research about the behaviour of its solutions for different values of the parameters involved. Indeed, mathematicians studied its invariants, boundedness, persistence, convergence, stability or its oscillatory behaviour (for example, see \cite{JanKLS95}, \cite{JanKLT98} and \cite{Feuer03}, as well as \cite{GroveL05} for a general amount of information).

Encouraged by those results, other max-type difference equations were studied. It should be highlighted the reciprocal difference equation with maximum, see \cite{Ladas96}, namely $$ x_{n+1} = \max \left\{ \frac{A_0}{x_n}, \dots, \frac{A_k}{x_{n-k}} \right\}, $$

\noindent where the coefficients $A_j$, $j=0,\ldots,k$, are positive real numbers, which stands out for the eventual periodicity of its solutions (for example, see \cite{CranstonK13} and \cite{BidwellF08}). Moreover, modifications of the reciprocal difference equation have been made by considering powers, periodic coefficient or changes of variables (\cite{Stevic08}, \cite{Stevic09}). Furthermore, the idea of max-type difference equations has been even generalized to rank-type difference equations, see \cite{Sauer11} and \cite{BerryS12}. 

This type of equations have applications in automatic control theory, for instance, when the law of regulation of certain systems depends on the maximum values of some regulated state parameters over certain time intervals. According to \cite{Bainov}, in \cite{Popov} we can find a model of a system for regulating the voltage of a generator of constant current with parallel simulation and the regulated quantity was the voltage at the source electric current. The equation describing the work of the regulator involves the maximum of the unknown function ant it has the form: $$T_0u'(t) + u(t) + q\cdot \max_{s\in[t-h,t]} u(s) = f(t), $$
 
\noindent where $T_0$ and $q$ are constants characterizing the object, $u(t)$ is the regulated voltage and $f(t)$ is the perturbed effect. Notice that the discretization of the problem yields to a max-type difference equation.    

The reader interested in this type of difference equations, as well as different characteristics of its dynamics,  may consult the monograph \cite{GroveL05}.

In this paper, we focus our attention on the possible periodic character of the solutions of 
\begin{equation}\label{Eq:Gk}
	x_{n+k}=\max\{x_{n+k-1},x_{n+k-2},\ldots, x_{n+1},0\}-x_n,
\end{equation}
where the initial conditions are real numbers.

First, let us mention that one of the first appearances of Equation~(\ref{Eq:Gk}) occurred in \cite{Golomb92}, where M. Golomb raised the problem of proving that if the initial conditions are monotonic, then the sequence $\left(x_n\right)_n$ -generated by iteration of (\ref{Eq:Gk})- is periodic of period $3k-1$. The solution was provided by David Callan among others. Notice that if $k=2$ or $k=3$, then the corresponding equations are globally periodic, of periods $5$ and $8$, respectively, that is, all the solutions are periodic and the least common multiple of its periods are $5$ and $8$, respectively.   
	
Next, some properties and generalities for Equation~(\ref{Eq:Gk}) were established by Cs\"{o}rnyei and Laczkovich in \cite{CsLa2001}. For instance, they prove that every solution of Equation (\ref{Eq:Gk}) is bounded and that the equation is not globally periodic, that is, it is possible to find appropriate initial conditions for which the generated sequence is not periodic. Nevertheless, they mention that the set of periods is unbounded. 
							
As a consequence, a natural question arises: to analyze the structure of the set of periods. As a first step, being already known the situation for the cases $k=2$ and $k=3$, we restrict our attention to $k=4$, 
\begin{equation}\label{Eq:G4}
	x_{n+4}=\max\{x_{n+3},x_{n+2}, x_{n+1},0\}-x_n.
\end{equation}
The main goal of this paper is to provide a complete description of the set of periods of~(\ref{Eq:G4}). 
Let us mention that in \cite{Barb} a few aspects on the periodicity of this recursion were established. The authors comment that for $k\geq 4$, the situation ``becomes more complicated, and the sequence could be periodic with a variety of periods, or nonperiodic''. Moreover, they remark that, in the region $0\leq x_1\leq x_2\leq \ldots\leq x_k$, the sequence has a ``universal period'', that is, $3k-1$. But, as far as we know, the question of determining the elements of the set of periods remained unsolved. In this work, we will prove that the set of periods of Equation~(\ref{Eq:G4}) is precisely
$$\left\{1,8,11\right\}\bigcup \left\{10\cdot a + 11\cdot b \ | \gcd(a,b)=1, a\geq 1, b\geq 2a+1 \right\}.$$
Additionally, apart from the fact that the set of periods is unbounded,  we will prove that $1674$ is the biggest natural number not included in it.

The paper is organized as follows. In Section~\ref{S:pre}, we present the basic notions and give some preliminary results, both for the general Equation (\ref{Eq:Gk}) and for the particular (\ref{Eq:G4}). Next, in Section~\ref{S:PerF4}, we will focus on Equation~(\ref{Eq:G4}), and  show the different possible ways in order to obtain periodic sequences, where the strategy consists of translating  the problem into a graph framework.  Jointly with the description of the periodic orbits, we obtain the set of periods of (\ref{Eq:G4}), as well as possible realizations of such periods. Then, a question arises, to determine whether there exists some natural number $N$ such that $n$ is automatically a period if $n\geq N$. This is outlined in Section~\ref{S:Maximum}, where such a number is found (the complete proof of the result can be consulted in \cite{LiNi}). Finally, in the last section we present some comments and several open problems for the corresponding generalized Equation (\ref{Eq:Gk}).
				
\section{Preliminaries}\label{S:pre}
In general, let $x_{n+k}=f(x_{n+k-1},\ldots,x_{n+1},x_{n})$ be an autonomous difference equation of order $k$, where $f:\Omega\subseteq X^k\rightarrow X$ is defined on some subset $\Omega$ of a finite Cartesian product of a set $X$. A solution $(x_n)$ is called periodic if $x_{n+m}=x_n$ for all $n\geq 1$ and some $m\in \mathbb{N}:=\{1,2,3, \ldots\}$. The smallest of such values $m$ is called the period of the solution. If there exists $n_0\geq 1$ such that $x_{n+m}=x_n$ for all $n\geq n_0$ and $m \in \mathbb{N}$, we say that the sequence is eventually periodic. 

For the particular case of~Equation (\ref{Eq:Gk}), we use $\mathrm{Per}(F_k)$ to denote its set of periods, $k\geq 2$. Our main task in this paper deals with the search of this set of periods. As a precedent, it is a well-known fact (see \cite{Barb}) that  $\mathrm{Per}(F_2)=\{1,5\}$, and $\mathrm{Per}(F_3)=\{1, 2, 8\}$. For~(\ref{Eq:G4}) the situation changes, the solutions are not necessarily periodic as was shown in \cite{CsLa2001}. Even, in the same reference, the unbounded character of $\mathrm{Per}(F_4)$ was proved. However, nothing is said about a complete description of the set of periods.

Before concentrating in Equation~(\ref{Eq:G4}), we give some properties for the general case. After this, our task will be focused on the case $k=4$. 

\subsection{Properties for the general case} The first observation which merits our attention is the fact that every eventually periodic sequence generated by Equation~(\ref{Eq:Gk}) is, in fact, periodic. To this end, realize that 
\begin{equation}\label{Eq:GkInversa}
x_n=\max\{x_{n+k-1},x_{n+k-2},\ldots, x_{n+1},0\}-x_{n+k}
\end{equation} can be viewed as the inverse difference equation linked to Equation~(\ref{Eq:Gk}). Or, in other terms, notice that 
$$F(x_1,x_2,\ldots,x_{k-1},x_k):=\left(x_2,x_3,\ldots,x_k,\max\left\{x_2,x_3,\ldots,x_k,0\right\}-x_1\right)$$ is a bijective map from $\mathbb{R}^k$ into itself, whose inverse is given by $$F^{-1}(y_1,y_2,\ldots,y_{k-1},y_k)=\left(\max\left\{y_1,y_2,\ldots,y_{k-1},0\right\}-y_k,y_2,y_3,\ldots,y_k\right).$$ 

\begin{proposition} \label{Prop_eventually}
Every eventually periodic sequence of Equation~(\ref{Eq:Gk}) is periodic.
\end{proposition}

The following result is immediate and tells us that there exists a (unique) fixed point for~(\ref{Eq:Gk}). 
\begin{proposition} \label{Prop_equi}
The unique equilibrium point of~(\ref{Eq:Gk}) is $\overline{x}=0$. In particular, $1\in\mathrm{Per}(F_k)$ for all $k\geq 2$.
\end{proposition}

In general, we can give a few results for the periods of~(\ref{Eq:Gk}).

\begin{proposition}\label{Remark_Per2} For Equation~(\ref{Eq:Gk}), it holds:
	\begin{itemize}
		\item[(a)] $2k\in\mathrm{Per}(F_k)$ for all $k\geq 4$.
		\item[(b)] For $k\geq 3$, $2\in\mathrm{Per}(F_k)$ if and only if $k$ is odd; a sequence $(x_n)$ has period $2$ if and only if for some $a>0$, either $x_{2j-1}=a$ and $x_{2j}=0$ for all $j\geq 1$, or $x_{2j-1}=0$ and $x_{2j}=a$ for all $j\geq 1$.  
	\end{itemize}
\end{proposition}

\begin{proof}
(a) Consider $\left(x_1,x_2,\ldots,x_k\right)=(0,x,0,x,x,\ldots,x)$, with $x>0$  as the initial conditions. Then, applying Equation~(\ref{Eq:Gk}) we obtain $0,x,0,x,x,\ldots,x,$ $\,\hspace{1mm}  x,0,x,0,0,\ldots, 0,$ $\,\hspace{1mm} 0, x, 0, x, x,\ldots, x, \dots, $
as can be easily checked. Being $(x_n)$ a periodic sequence, denote by $q$ its period (observe that $q|2k$). Obviously, $q\notin\{1,2,k\}$ as $x>0$. Even, $q\neq k-1$, because $(k-1)\nmid 2k$ for $k\geq 4$. If $q=2k$, we finish the proof. Otherwise, since $q \leq k-2$ and $(x_{4},x_5,\ldots, x_k,x_{k+1})=(x,x,\ldots,x,x)$ is a string of the periodic sequence with length $k-2$, we would deduce that, in fact, $(x_n)=(x)_n$, which would give $x=0$, contrarily to our hypothesis.  Therefore, $q=2k$.  
							
(b) If $k$ is odd, $k\geq 3$, take initial conditions $(x,0,x,0,\ldots,x,0,x,0,x),$ with $x>0$. Then it is straightforward to check that these initial conditions generate a periodic sequence $(x_n)$ of order $2$. 

If $k=2m$ is even, then $2\not\in\mathrm{Per}(F_{2m})$. For if $(x_1,x_2,\ldots,x_{2m-1},x_{2m})=(a,b,\ldots,a,b)$ are initial conditions providing the two-periodic sequence, with $a\geq b$, then $a=x_1=x_{2m+1}=\max\{a,b,0\}-a=\max\{a,0\}-a$, so $2a=\max\{a,0\}$, which implies $a=0$; similarly, $b=x_2=x_{2m+2}=\max\{a,b,0\}-b=0-b=-b,$ hence $b=0$. Therefore $a=b=0$ and we would obtain the equilibrium point $\overline{x}=0$, a contradiction.  

The proof of the characterization of two-periodic sequences is immediate and we omit it.
\end{proof}

\begin{proposition}\label{P:maxNoNeg}
Let $\left(x_n\right)_n$ be a periodic solution of Equation~(\ref{Eq:Gk}) of period $p$. Then $m=\max\left\{x_j:j=1,\ldots,p\right\} \geq 0$. Additionally, if $m=0$, then the solution is $\mathbf{\overline{0}}=(0,0,0,0,\ldots)$.
\end{proposition}
				
\begin{proof}
Without loss of generality, reordering the periodic sequence if necessary, we can assume that $x_1=\max\left\{x_j:j=1,\ldots,p\right\}$. Suppose that $x_1<0$ and derive a contradiction. Since $x_1=\max\left\{x_p,x_{p-1},\ldots,x_{p-(k-2)},0\right\}-x_{p-(k-1)}$, where the indexes are taken $\mathrm{mod} (p)$ if $p<k$, and $$\max\left\{x_p,x_{p-1},\ldots,x_{p-(k-2)},0\right\}=0,$$ we deduce that $x_1+x_{p-(k-1)}=0$, which is not possible from the fact that all the sequence is negative.
It is immediate to check that if $m = 0$, then necessarily $(x_n)=\mathbf{\overline{0}}.$
\end{proof}

From now on, for the periodic sequence $(x_n)$ of period $p$, we will assume that $x_1=\max\left\{x_j:j=1,\ldots,p\right\}\geq 0$.
				
\begin{proposition}\label{P:Positivos}
Let $\left(x_n\right)_n$ be a periodic solution of Equation~(\ref{Eq:Gk}) of period $p$, $p\geq 2$. Then 
$$x_j\geq 0, \,\, \text{for} \,\, j=1,\ldots,k, \ \ \ \text{and} \ \ \ x_{p-j}\geq 0, \,\, \text{for} \,\, j=0,\ldots,k-2.$$	Additionally, $x_{k+1}\leq 0$ and $x_{p-k+1}\leq 0.$ (All the indexes are taken $\mathrm{mod}(p)$.)
\end{proposition}

\begin{proof}
By periodicity of the sequence, for $j=2,\ldots,k$, we find 
\begin{eqnarray*}
x_{j}&=&\max\{x_{j-1},\ldots,x_1,x_{p},x_{p-1},\ldots,x_{p-(k-j-1)},0\}-x_{p-(k-j)}\\
&=& x_1 - x_{p-(k-j)} \geq 0.
\end{eqnarray*}
For the second set of inequalities, we will apply again the periodicity of the sequence going backward by Equation (\ref{Eq:GkInversa}). For instance, 
$$x_p=\max\left\{x_1,x_2,\ldots,x_{k-1},0\right\}-x_k=x_1-x_k\geq 0,$$ $$x_{p-k+2}=\max\left\{x_{p-k+3},x_{p-k+4}, \ldots,x_{p},x_{1},0\right\}-x_2=x_1-x_2\geq 0.$$
Finally, $$x_{k+1}=\max\left\{x_k,x_{k-1},\ldots,x_{2},0\right\}-x_1\leq 0,$$ $$x_{p-k+1}=\max\left\{x_{p-k+2},x_{p-k+3},\ldots,x_{p},0\right\}-x_1\leq 0.$$
\end{proof}

\begin{corollary} \label{lemma0}
Let $(x_n)$ be a $p$-periodic sequence of (\ref{Eq:G4}). If there exists a $j \in \{1,...,p\}$ such that $x_j = x_1$, then $x_{j+1}, x_{j+2}$ and $x_{j+3}$ are non-negative and $x_{j+4}$ is non-positive.
\end{corollary}

Notice that if we multiply the initial conditions of a periodic sequence by a positive constant $\alpha$, the sequence is still periodic with the same period. Indeed, the proof follows directly by taking out the factor $\alpha$ in Equation $(\ref{Eq:Gk})$ and will be omitted. 

\begin{proposition} \label{Prop_scalar}
Let $(x_n)$ be a periodic solution of Equation $(\ref{Eq:Gk})$ generated by the initial conditions $(x_1,\ldots,x_k)$. Then the sequence $(\alpha \cdot x_n)_n$ with $\alpha>0$ is periodic with the same period.
\end{proposition}

Finally, it is relevant to mention that every solution of Equation (\ref{Eq:Gk}) is bounded. This result was proved in \cite[Th. 12]{CsLa2001}.
\begin{proposition} \label{Boundedness}
Every solution $(x_n)$ of Equation (\ref{Eq:Gk}) is bounded.
\end{proposition}

\subsection{Properties for $k=4$}

From now on, we leave the general recursion~(\ref{Eq:Gk}) and focus our attention on the corresponding difference equation of order $k=4$,~(\ref{Eq:G4}): $x_{n+4}=\max\{x_{n+3},x_{n+2}, x_{n+1},0\}-x_n.$
				
\begin{proposition} \label{Prop_monotonic}
Let $(x_n)$ be a solution of Equation (\ref{Eq:G4}). If there exist four consecutive monotonic terms, then the solution is an 11-cycle.
\end{proposition}

\begin{proof}
Firstly, notice that if the four consecutive monotonic terms are the initial conditions, the result follows by \cite{Golomb92}. Otherwise, we will have an eventually periodic sequence, but by Proposition \ref{Prop_eventually}, it will be periodic. 
\end{proof}

In the case of $8$-cycles, we present the following characterization.
\begin{proposition}\label{P:periodo8}
Given Equation~(\ref{Eq:G4}) of order $k=4$, a sequence $\left(x_n\right)_n$ is periodic of period $8$ if and only if 
\begin{equation} \label{Eq:8p}
\left(x_n\right)_n=\left(\ldots, x, 0, x, \alpha, 0, x, 0, x-\alpha, x, 0, x, \alpha, 0, x, 0, x-\alpha, \ldots\right),
\end{equation}
with $x>0$ and $\alpha\in[0,x]$.
\end{proposition}

\begin{proof}
The sufficiency can be shown in a direct way as follows. Without loss of generality, suppose that $(x_1,x_2,x_3,x_4)=(x,0,x,\alpha),$ with $x>0$ and $\alpha\in[0,x]$. It is immediate, by the computation of the following terms, $x_j, \ j=5,\ldots, 12$, that the sequence has a period which is a divisor of $8$.

Since $x\neq 0$, period $1$ is excluded. Also, period $2$ is excluded by Proposition~\ref{Remark_Per2}, and if the period were $4$, then $x=x_1=x_5=0$, which is a new contradiction. Therefore, the sequence has exactly period $8$.
							
Now, suppose that $(x_n)$ is a periodic sequence of period $8$.  Notice that, by Proposition \ref{P:maxNoNeg}, $x_1 = \max\{x_n: n\geq 1 \} > 0$. Also, by Proposition~\ref{P:Positivos}, $x_1,x_2,x_3$ and $x_4$, as well as $x_6, x_7$ and $x_8$, are non-negative terms, while $x_5\leq 0$. Even more, we have 
\begin{equation}\label{Eq:tresen8}
x_8=x_1-x_4; \ \ x_7=x_1-x_3; \ \ x_6=x_1-x_2.
\end{equation}

Since $x_8=\max\{x_7,x_6,x_5,0\}-x_4=x_1-x_4,$ we deduce that $x_1=\max\{x_7,x_6,x_5,0\},$ with $x_5\leq 0$.
\begin{itemize}
\item[(i)] If $x_1=x_6$, from (\ref{Eq:tresen8}) we deduce that $x_2=0$. Also, $x_5=x_1-x_1=0$, and our periodic sequence has the form $\left(x,0, \cdot, \cdot, 0,x,\cdot,\cdot\right).$ On the other hand, $0=x_5=\max\{x_4,x_3,x_2,0\}-x_1=\max\{x_4,x_3\}-x_1,$ that is, $x_1=\max\{x_4,x_3\}.$ If $x_1=x_3$, then $x_7=0$ according to (\ref{Eq:tresen8}); in this case the choice of $x_4$ is arbitrary, with $0\leq x_4\leq x_1$, and $x_8=x_1-x_4$; if we put $\alpha=x_4\in[0,x_1]$, we obtain the sequence 
$$\left(x_1,0, x_1, \alpha, 0,x_1, 0,x_1-\alpha,\ldots\right),$$ and it is easily seen that we generate a periodic sequence of period $8$. If, otherwise, $x_1=x_4$, now $x_8=0$, and if we set an arbitrary $\beta\in[0,x_1]$, we obtain the $8$-periodic sequence 
$$\left(x_1,0, \beta, x_1, 0,x_1, x_1-\beta,0\ldots\right).$$
								
\item[(ii)] If $x_1=x_7$, (\ref{Eq:tresen8}) gives $x_3=0$. Similarly to case (i), $x_5=0$ and 
we have $x_1=\max\{x_2,x_4\}.$ If $x_1=x_2$, again (\ref{Eq:tresen8}) gives $x_6=0$, and taking into account that $x_6=\max\{x_5,x_4,x_3,0\}-x_2=x_4-x_1=0$, we deduce that $x_4=x_1$, $x_8=x_1-x_4=0$ and arrive to 
$$\left(x_1,x_1, 0, x_1, 0,0,x_1, 0,\ldots\right).$$
If, on the contrary, $x_1=x_4$, from (\ref{Eq:tresen8}) $x_8=0$, and as above $x_5=0$. If we put $x_2=\gamma\in[0,x_1]$, then $x_6=x_1-\gamma$ and we obtain the periodic sequence of period $8$  
$$\left(x_1,\gamma, 0, x_1, 0, x_1-\gamma, x_1, 0,\ldots\right).$$
\end{itemize}
\end{proof}

In fact, in the above proposition, we have implicitly found the unique periodic solutions whose terms are all non-negative. 
\begin{proposition}\label{P:todosPositivos}
Let $(x_n)$ be a periodic sequence of Equation~(\ref{Eq:G4}), with $x_n\geq 0$ for all $n\geq 0$. Then, either the sequence is $\mathbf{\overline{0}}=(0,0,0,0,\ldots)$ or it is a periodic sequence of period $8$ given by~(\ref{Eq:8p}).
\end{proposition}
\begin{proof}
By Proposition~\ref{P:Positivos}, we know that $x_j\geq 0$ for $j=1,2,3,4$ and $j=p-2,p-1,p$. Moreover, $x_5\leq 0$. Since we suppose that $x_j\geq 0$ for all $j$, we deduce that $x_5=0$. From the recurrence~(\ref{Eq:G4}), we have that $x_1=\max\{x_2,x_3,x_4\}.$ We now distinguish the corresponding cases.
							
(i) If $x_1=x_2$, it turns to $x_6=\max\{x_5,x_4,x_3,0\}-x_2=\max\{x_4,x_3\}-x_1$, and by the choice of $x_1,$ again $x_6\leq 0$, that is, $x_6=0$ from our hypothesis. Therefore, either $x_3=x_1$ or $x_1=x_4$. When $x_1=x_2=x_3$, necessarily $x_4=x_1$, otherwise $x_7=\max\{x_6,x_5,x_4,0\}-x_3=x_4-x_1<0,$ contrary to our initial assumption; thus, putting $x_1=x$ the initial conditions are $(x,x,x,x),$ with $x\geq 0$, and the sequence generated is $\left(x,x,x,x,0,0,0,-x,0,\ldots\right)$, consequently $x=0$ and we have the equilibrium point $\overline{\textbf{0}}=(0,0,0,0,\ldots)$. And for $x_1=x_2=x_4$, with $0\leq x_3\leq x_1$,  putting $x_1=x, x_3=y$ we find the sequence $\left(x,x,y,x, 0, 0, x-y, -y, \ldots\right),$ so $y=0$ and the recurrence provides an $8$-periodic sequence as can be easily verified.
							
(ii) If $x_1=x_3,$ apart from $x_5=0$ we have that  $x_6=\max\{x_5,x_4,x_3,0\}-x_2=x_3-x_2=x_1-x_2,$ and 
$x_7=\max\{x_6,x_5,x_4,0\}-x_3=\max\{x_1-x_2,x_4\}-x_1 \leq 0,$ so $x_7= 0$, that is, $\max\{x_1-x_2,x_4\}=x_1.$ Then, either $x_2=0$, and we obtain the $8$-periodic sequence generated by the initial conditions $(x,0,x,y),$ with $0\leq y\leq x$ (here $x=x_1$ and $x_4=y$); or $x_1=x_4$, and in this situation it can be checked without difficulty that $(x,y,x,x)$ (with $x_1=x, x_2=y$, $y\leq x$) generates the sequence $\left(x,y,x,x,0,x-y,0,-y,\ldots\right)$, and as in case (i), $y=0$ and the initial conditions yield an $8$-periodic sequence. 
							
(iii) If $x_1=x_4$, then it is immediate to obtain $x_6=x_1-x_2$ and $x_7=x_1-x_3$. Consequently, $x_8=\max\{x_1-x_3,x_1-x_2\}-x_1\leq 0,$ thus $x_8=0$ and $x_1=\max\{x_1-x_3,x_1-x_2\}$. If $x_2=0$, from $(x_1,0,x_3,x_1)$ we obtain a periodic sequence of period $8$; if $x_3=0$, now we obtain the $8$-periodic sequence $(x_1,x_2,0,x_1,0, x_1-x_2, x_1, 0, x_1, x_2, 0, x_1, \ldots)$. 
\end{proof}

In the next result, we prove that $1, 8, 11$ are the first periods in $\mathrm{Per}(F_4)$.
\begin{proposition}\label{P:1,8,11}
It holds $\mathrm{Per}(F_4) \cap [1,11]=\left\{1,8,11\right\}.$
\end{proposition}

\begin{proof}
We know that $\{1,8,11\}\in \mathrm{Per}(F_4)$: 
the sequence $\mathbf{\overline{0}}=(0,0,0,\ldots)$ has period $1$; the initial conditions $(x,0, x, \alpha)$, with $x>0, \alpha\in[0,x]$, generate a $8$-cycle; and if the initial conditions are monotonic, with $x_1\neq 0$, then the sequence $(x_n)$ has period $11$ according to Proposition~\ref{Prop_monotonic}. 

Next, suppose that $(x_n)$ is a $p$-periodic sequence of period $p\leq 10.$ Without loss of generality, assume that 
$x_1=\max\{x_j:j\geq 1\}>0$. 
\begin{itemize}
\item If $p\leq 7$, by Proposition~\ref{P:Positivos} we find that $x_1,\ldots, x_p$ are non-negative. So, by Proposition~\ref{P:todosPositivos}, either $p=1$ or $p=8$. Thus,  $p=1$. 
\item If $p=9$, Proposition~\ref{P:Positivos} gives  $x_5\leq 0, x_6\leq 0$, and  $x_j\geq 0$ for $j=1,2,3,4,7,8,9$. Moreover, by the recurrence,
\begin{equation}\label{Eq:pas1}
x_8=\max\{x_7,x_6,x_5,0\}-x_4=x_7-x_4,
\end{equation}
\begin{equation}\label{Eq:pas2}
x_7=\max\{x_6,x_5,x_4,0\}-x_3=x_4-x_3.
\end{equation} 
  From (\ref{Eq:pas1}) and (\ref{Eq:pas2}) we deduce $x_8=x_7-x_4=(x_4-x_3)-x_4=-x_3$, so $x_3+x_8=0$ with $x_3,x_8\geq 0$, which implies $x_3=x_8=0$. On the other hand, if we consider the recurrence acting backward, we have $x_8=\max\{x_9,x_1,x_2,0\}-x_3=x_1-x_3$, and consequently $x_1=0$, so the sequence is $\mathbf{\overline{0}}$, a contradiction. 
	\item If $p=10$, now Proposition~\ref{P:Positivos} yields $x_j\geq 0$ for $j\in\{1,2,3,4,8,9,10\}$, whereas $x_5\leq 0, x_7\leq 0$. 
	From $x_7=\max\{x_6,x_5,x_4,0\}-x_3=\max\{x_6,x_4\}-x_3\leq 0$, we deduce that $\max\{x_6,x_4\}\leq x_3,$ in particular, 
	$x_4\leq x_3.$ By periodicity acting backward, $x_4=\{x_3,x_2,x_1,0\}-x_{10}=x_1-x_{10},$ so $x_{10}=x_1-x_4$. Similarly, 
		$x_9=x_1-x_3$ and $x_8=x_1-x_2$. By using that $x_5\leq 0$, and $x_5=\max\{x_6,x_7,x_8,0\}-x_9=\max\{x_6,x_8\}-x_9$, we deduce 
		that $x_8\leq x_9$, or equivalently, $x_3\leq x_2$. Therefore, we have 	$x_4\leq x_3 \leq x_2$ and $x_2\leq x_1$. 
		From Proposition~\ref{Prop_monotonic} we conclude that $(x_n)$ is $11$-periodic, a contradiction.
\end{itemize}
\end{proof}

Finally, we will define an equivalence relation in $\mathbb{R}^4$, which will be very useful in the sequel. Notice that, by Equation (\ref{Eq:GkInversa}), for given initial conditions, we can build a unique sequence $(x_n)_{n\in\mathbb{Z}}$. 

\begin{definition} \label{Def_equivalence}
Let $\mathbf{x}, \mathbf{y} \in \mathbb{R}^4$. We will say that $\mathbf{x}=(x_1,x_2,x_3,x_4) \sim \mathbf{y}=(y_1,y_2,y_3,y_4)$ if and only if $\mathbf{x}$ and $\mathbf{y}$ generate under Equations (\ref{Eq:G4}) and (\ref{Eq:GkInversa}) the same sequences $(x_n)_{n\in\mathbb{Z}}$ and $(y_n)_{n\in\mathbb{Z}}$ up to a shift. In particular, $\mathbf{x}\sim \mathbf{y}$ if $\mathbf{x}$ and $\mathbf{y}$ generate the same periodic sequence under Equation (\ref{Eq:G4}).
\end{definition}

Notice that $\sim$ is an equivalence relation. For instance, $(x,y,z,y) \sim (x,z,z,y)$ with $x > y > z > 0$, since under Equation (\ref{Eq:G4}) the tuple evolves as follows: $$x, y, z, y, y-x, 0, y-z, -z, x-z, x-z, x-y, x, z, z, y.$$

\section{Description of periodic sequences. Characterization of the set of periods $\mathrm{Per}(F_4)$} \label{S:PerF4}
In this section, we deeply analyze the possible form of a periodic sequence of Equation~(\ref{Eq:G4}), or to better say, we try to locate the possible configurations of the initial conditions in order to obtain a periodic sequence. By carrying out this study we will be able to set the associate period of the periodic sequence.
We divide this section as follows. The first part will be devoted to study the movement of a tuple of initial conditions $(x_1,x_2,x_3,x_4)$ with $x_1 = \max\{x_n:n\geq 1\}$ that generate a periodic sequence; we will see that they can be divided in five Cases $C_i, i=1, \ldots, 5$, and that the orbit of a solution visits these cases in a concrete way provided that the initial conditions verify a particular property that will be called Condition U. The study of the orbit's movement by the different cases will provide us a form of computing the period of the solution (by adding blocks of ten or eleven elements) in most situations. In the second part, Subsection \ref{Sec_controversial}, we will start analysing what initial conditions yield to periodicity between cases, that is, whenever periodicity holds in the middle of the corresponding block of ten or eleven iterations that goes from a case $C_i$ into another $C_j$. We will find these tuples, that we will call controversial cases, and by using Definition \ref{Def_equivalence} we will reduce them into two classes of equivalence, namely $(x,y,0,z)$ and $(x,z,y,0)$, where $x\geq y \geq z \geq 0$ and $x>z$. Additionally, in Subsection \ref{Sec_specific}, we will determine the period of the two classes mentioned above. Next, we will focus on the initial conditions that verify more than one case $C_i$. For instance, $(x,y,z,y)$ with $x>y>z\geq 0$ satisfies $C_1$ and $C_4$. These tuples, which form the intersection between cases, will be analysed in detail in Subsection \ref{Sec_intersection}, and we will find their corresponding periods. After that, Subsection \ref{Sec_diagram} is devoted to the study of the remaining possibility, i.e., the initial conditions belong to a unique case $C_i$ and the movement of those initial conditions is directed by Diagram \ref{Diagrama} (see Figure \ref{Diagrama} described in the forthcoming Subsection \ref{Sec_routes}), starting and ending in the same case $C_i$ by the corresponding blocks of ten or eleven iterations shown in Proposition \ref{prop_diag} (we will call this property Condition U).
Finally, we gather in Subsection \ref{Sec_main} the study of the previous possibilities in the form of a main theorem, Theorem \ref{Th_PerF4}, on the set of periods of Equation~(\ref{Eq:G4}).

\subsection{The routes of periodic solutions and their periods} \label{Sec_routes}
From now on, let us suppose that the solution $(x_n)$ is $p$-periodic with $p\geq 12$. The following result establishes that, if $x_1$ is the biggest element of the solution, $(x_n)$, then indefinitely we will obtain again $x_1$ after ten or eleven iterations.
\begin{proposition} \label{prop_diag}
Let $(x_n)$ be a periodic sequence of period $p$, with $x_1=\max\{x_j: 1\leq j \leq p\}$. Then either $x_{11}=x_1$ or $x_{12}=x_1$. In both cases, $x_j, x_{j+1}, x_{j+2}, x_{j+3}$ are non-negative for $j=11$ in the first case and for $j=12$ in the second one.
\end{proposition}

\begin{proof}
By Proposition~\ref{P:Positivos}, we know that $x_2,x_3,x_4$ are non-negative. Realize that if $x_1\geq x_2\geq x_3\geq x_4$, we have an $11$-cycle or the equilibrium point if $x_1=0$. From now on, we assume that $x_1>0$ and distinguish several cases:
\begin{itemize}
\item[(i)] Suppose $x_1 \geq x_2 \geq x_4\geq x_3\geq 0$. Then, by iterating under (\ref{Eq:G4}), 
\begin{eqnarray*}
	x_5&=&x_2-x_1 \leq 0,\\
	x_6&=& \max\left\{x_5,x_4,x_3,0\right\}- x_2=x_4-x_2 \leq 0, \\
	x_7&=& \max\left\{x_6,x_5,x_4,0\right\}- x_3=x_4-x_3,\\
	x_8&=& \max\left\{x_7,x_6,x_5,0\right\}- x_4=x_7-x_4=-x_3 \leq 0,\\
	x_9&=& \max\left\{x_8,x_7,x_6,0\right\}- x_5=x_7-x_5=x_4-x_3-x_2+x_1,\\
	x_{10}&=& \max\left\{x_9,x_8,x_7,0\right\}- x_6=x_9-x_6= x_1-x_3,\\            x_{11}&=& \max\left\{x_{10},x_9,x_8,0\right\}- x_7= x_{10}-x_7 = x_1-x_4,\\
	x_{12}&=& \max\left\{x_{11},x_{10},x_9,0\right\}- x_8 = x_1-x_3+x_3=x_1.
\end{eqnarray*}
Moreover, we also find that the following three terms are non-negative: $x_{13} = x_2 + x_3 - x_4; \ \ x_{14} = x_3; \ \ x_{15} = x_4.$ 
				
\item[(ii)] Let $x_1 \geq x_3 \geq x_4\geq x_2\geq 0$. In this case, 
$$x_5=x_3-x_1 \leq 0, \ \ x_6 = \max\left\{x_5,x_4,x_3,0\right\}- x_2=x_3-x_2, $$
and $x_7=\max\left\{x_6,x_5,x_4,0\right\}- x_3=\max\left\{x_3-x_2,x_4,0\right\}-x_3.$

-- If additionally $x_3\geq x_2+x_4,$ then we continue
\begin{eqnarray*}
	x_7&=&\max\left\{x_6,x_5,x_4,0\right\}- x_3=x_3-x_2-x_3=-x_2\leq 0,\\
	x_8&=& \max\left\{x_7,x_6,x_5,0\right\}- x_4=x_6-x_4=x_3-x_2-x_4 \geq 0,\\
	x_9&=& \max\left\{x_8,x_7,x_6,0\right\}- x_5=x_3-x_2-x_3+x_1=x_1-x_2,\\
	x_{10}&=& \max\left\{x_9,x_8,x_7,0\right\}- x_6=x_1- x_2-x_3+x_2=x_1-x_3,\\    x_{11}&=& \max\left\{x_{10},x_9,x_8,0\right\}- x_7= x_{9}-x_7 = x_1-x_2+x_2=x_1;
\end{eqnarray*}
even more, the next three terms are also non-negative (recall that $x_3\geq x_2+x_4$): $x_{12} = (x_1 - x_3) + (x_2 + x_4); \ \ x_{13} = x_2; \ \ x_{14} = x_3.$

-- If, in the contrary, $x_3\leq x_2+x_4,$ we will now obtain 
\begin{eqnarray*}
	x_7&=&\max\left\{x_6,x_5,x_4,0\right\}- x_3=x_4-x_3 \leq 0,\\
	x_8&=& \max\left\{x_7,x_6,x_5,0\right\}- x_4=x_6-x_4=x_3-x_2-x_4 \leq 0,\\
	x_9&=& \max\left\{x_8,x_7,x_6,0\right\}- x_5= x_6-x_5=x_3-x_2-x_3+x_1=x_1-x_2,\\
	x_{10}&=& \max\left\{x_9,x_8,x_7,0\right\}- x_6=x_9-x_6=x_1-x_2-x_3+x_2=x_1-x_3,\\                    
	x_{11}&=& \max\left\{x_{10},x_9,x_8,0\right\}- x_7= (x_1-x_4)+(x_3-x_2)\geq 0,\\
	x_{12}&=& \max\left\{x_{11},x_{10},x_9,0\right\}- x_8= x_{11}-x_8 =x_1\geq 0,
\end{eqnarray*} $x_{13} = x_2, x_{14} = x_3$ and $x_{15} = x_4 + x_2 - x_3$ are non-negative.

\item[(iii)] Let $x_1 \geq x_3 \geq x_2\geq x_4\geq 0$. The first terms of the sequence are:
\begin{equation*}
	x_5=x_3-x_1 \leq 0,\, \hspace{5mm} x_6=x_3-x_2,\, \hspace{5mm} x_7=\max\left\{x_3-x_2,x_4\right\}-x_3.
\end{equation*}

Now, we must discern two possibilities, the procedure is similar, so we left in charge of the reader the checking of the computations: 

-- if $x_3\geq x_2+x_4,$ then $x_7 = -x_2\leq 0; \ \ x_8 = x_3-x_2-x_4; \ \  x_9 =  x_1-x_2; \ \  x_{10} = x_1-x_3; \ \  x_{11} = x_1,$ and the subsequent three terms are non-negative: $x_{12}= (x_1-x_3)+(x_2+x_4); \ \ x_{13}= x_2; \ \ x_{14}= x_3$.

-- if $x_3\leq x_2+x_4,$ and therefore $0\leq x_1-x_2-x_4+x_3\leq x_1$, we will have $x_7= x_4-x_3\leq 0; \ \ x_8= x_3-x_2-x_4 \leq 0; \ \ x_9= x_1-x_2; \ \ 
x_{10}= x_1-x_3;\ \ x_{11}= x_1-x_2-x_4+x_3; \ \ x_{12}=x_1;$ and $x_{13}= x_2; \ \ x_{14}= x_3; \ \ x_{15}= x_2+x_4-x_3. $
				
\item[(iv)] Consider $x_1 \geq x_4 \geq x_2\geq x_3\geq 0$. During the computation of the terms we have to use that $x_1-x_3\geq x_1+x_2-x_3-x_4\geq 0$, 
$x_2-x_3-x_4\leq 0$ and $x_1+x_2\geq x_3+x_4$. Then, the reader is encouraged to check that $x_5=x_4-x_1 \leq 0; \ \ x_6= x_4-x_2; \ \ x_7= x_4-x_3; \ \ x_8= -x_3 \leq 0; \ \ x_9= x_1-x_3; \ \ x_{10}= x_1+x_2-x_3-x_4; \ \ x_{11}= x_1-x_4; \ \ x_{12}= x_1$ and $x_{13}= x_3; \ \ x_{14}= x_4+x_3-x_2; \ \ x_{15} = x_4.$
 
 \item[(v)] Consider $x_1 \geq x_4 \geq x_3\geq x_2\geq 0$. It is easy to check that $x_5=x_4-x_1 \leq 0; \ \ x_6= x_4-x_2; \ \ x_7= x_4-x_3; \ \ x_8= -x_2 \leq 0; \ \ x_9= x_1-x_2; \ \ x_{10}= x_1-x_4; \ \ x_{11}= x_1-x_2-x_4+x_3; \ \ x_{12}= x_1$  and $x_{13}= x_2; \ \ x_{14}= x_4; \ \	x_{15}= x_2+x_4-x_3.$   

\end{itemize}
\end{proof}

By inspection of the different cases developed in the proof of Proposition \ref{prop_diag}, we observe that if the orbit were an 11-cycle, then there would exist four consecutive monotonic terms. For example in case (i), if the initial conditions generate an $11$-cycle, then $x_{13}=x_2$ implies $x_3=x_4$, and $(x_1,x_2,x_3,x_3)$ are four consecutive monotonic terms. Jointly with Proposition \ref{Prop_monotonic} we obtain:
				
\begin{corollary}
Let $(x_n)$ be a solution of Equation (\ref{Eq:G4}). Then the solution is an 11-cycle if and only if there exist four consecutive monotonic terms.
\end{corollary}

Notice that in cases (ii) and (iii) of the proof of Proposition~\ref{prop_diag}, the condition that establishes if $x_{11} = x_1$ or $x_{12} = x_1$ is the fact that $x_2 + x_4 \leq x_3$, or the reverse inequality, respectively. In this sense, we can sum up the cases for initial conditions in the following five:

\textbf{Case 1 (}$\mathbf{C_1}$\textbf{):} $x_1 \geq x_2 \geq x_4 \geq x_3$.

\textbf{Case 2 (}$\mathbf{C_2}$\textbf{):} $x_1 \geq x_3 \geq \max\{x_2, x_4\}$ with $x_3 \geq x_2 + x_4$.

\textbf{Case 3 (}$\mathbf{C_3}$\textbf{):} $x_1 \geq x_3 \geq \max\{x_2, x_4\}$ with $x_3 \leq x_2 + x_4$. 

\textbf{Case 4 (}$\mathbf{C_4}$\textbf{):} $x_1 \geq x_4 \geq x_2 \geq x_3$.

\textbf{Case 5 (}$\mathbf{C_5}$\textbf{):} $x_1 \geq x_4 \geq x_3 \geq x_2$.

In Figure \ref{Diagrama} we show the relationship between the different cases based on the analysis realised in Proposition \ref{prop_diag}. An arrow from Case $C_k$ to Case $C_l$ indicates that if we start with initial conditions $(x_1,x_2,x_3,x_4)$ holding the inequalities of Case $C_k$, then after $j=10$ or $j=11$ iterations the new  non-negative terms $(x_{j+1},x_{j+2},x_{j+3},x_{j+4})$ satisfy the conditions of Case $C_l$.

\begin{figure}[ht] 
\includegraphics[scale=0.45]{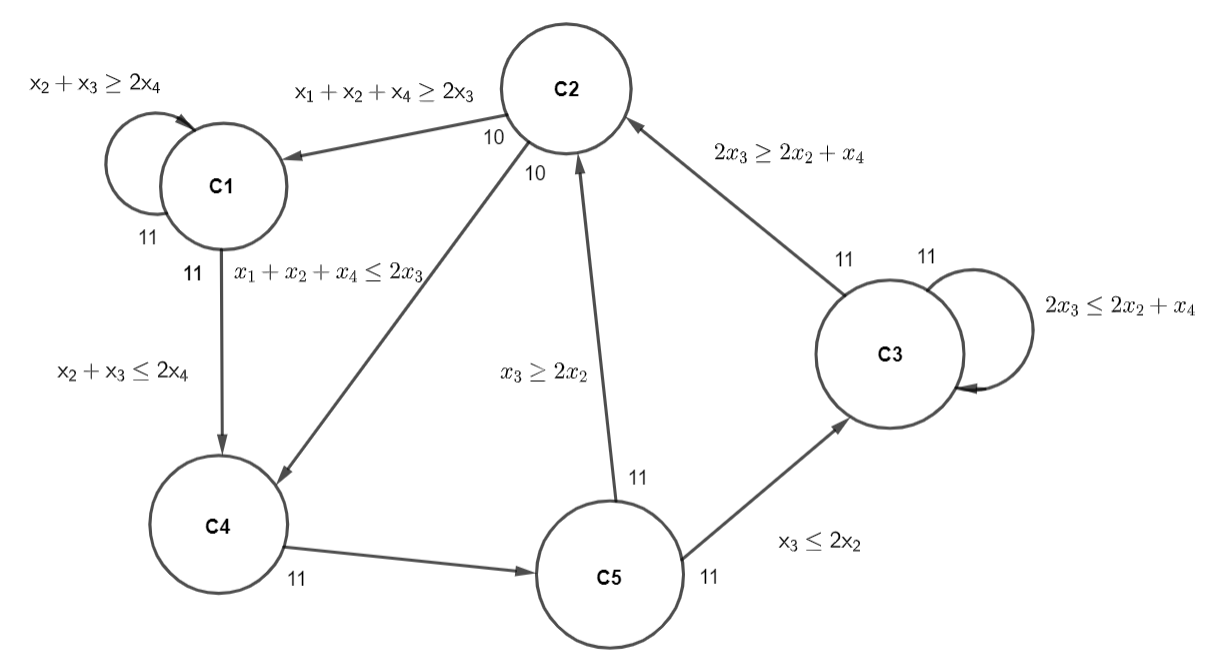}
\caption{The movement of Cases when we iterate Equation~(\ref{Eq:G4}).}
\label{Diagrama}
\end{figure}

Next, we summarize in Table \ref{Table_cases} how a tuple $(x_1,x_2,x_3,x_4)$ evolves under Equation (\ref{Eq:G4}) when it satisfies the conditions of each case. Notice that the information collected in the table follows by the inspection of the proof of Proposition \ref{prop_diag}.

\begin{table}[ht] 
\begin{tabular}{@{}cccc@{}}
\toprule
\textbf{$C_1:$} & $x_1 \geq x_2 \geq x_4 \geq x_3$ & $\xrightarrow{11}$ & $(x_{12}= x_1, x_2+x_3-x_4, x_3, x_4)$ \\ \midrule
\textbf{$C_2:$} & \begin{tabular}[c]{@{}c@{}}$x_1 \geq x_3 \geq \max\{x_2, x_4\}$\\ $x_3 \geq x_2 + x_4$\end{tabular}  & $\xrightarrow{10}$ & $(x_{11}= x_1, x_1-x_3+x_2+x_4, x_2, x_3)$ \\ \midrule
\textbf{$C_3:$} & \begin{tabular}[c]{@{}c@{}}$x_1 \geq x_3 \geq \max\{x_2, x_4\}$ \\ $x_3 \leq x_2 + x_4$\end{tabular} & $\xrightarrow{11}$ & $(x_{12}= x_1, x_2, x_3, x_2+x_4-x_3)$ \\ \midrule
\textbf{$C_4:$} & $x_1 \geq x_4 \geq x_2 \geq x_3$ & $\xrightarrow{11}$ & $(x_{12}= x_1, x_3, x_4+x_3-x_2, x_4)$ \\ \midrule
\textbf{$C_5:$} & $x_1 \geq x_4 \geq x_3 \geq x_2$ & $\xrightarrow{11}$ & $(x_{12}= x_1, x_2, x_4, x_2+x_4-x_3)$     \\ \bottomrule
\end{tabular}
\caption{Evolution of a tuple $(x_1,x_2,x_3,x_4)$ in the different cases.}
\label{Table_cases}
\end{table}

Once we have presented Table \ref{Table_cases}, let us justify the diagram of Figure~\ref{Diagrama}. If we start in $C_1$, notice that $x_{12} \geq x_{13}$, since $x_1 + x_4 \geq x_2 + x_3$. Now,
\begin{itemize}
    \item If $x_2 + x_3 \geq 2x_4$, then $x_{12} \geq x_{13} \geq x_{15} \geq x_{14}$ and we are in $C_1$.
    \item If $x_2+x_3 \leq 2x_4$, then $x_{12} \geq x_{15} \geq x_{13} \geq x_{14}$, since $x_3-x_4 \geq 0$. In this case, we are in $C_4$.
\end{itemize} 

Now, consider that the tuple satisfies $C_4$. As $x_4-x_2 \geq 0$, it follows that $x_{14} \geq x_{13}$. On the other hand, $x_3 - x_2 \leq 0$, so $x_{15} \geq x_{14}$. To sum up, $x_{12} \geq x_{15} \geq x_{14} \geq x_{13}$ and we are in $C_5$.

Next, if we start in $C_5$, $x_{12}\geq x_{14} \geq x_{13}$. Notice that $x_2 - x_3 \leq 0$, which implies $x_{14} \geq x_{15}$. So,
\begin{itemize}
    \item If $x_3 \geq 2x_2$, then $x_4 \geq 2x_2+x_4-x_3$, which implies $x_{14} \geq x_{13}+x_{15}$ and we are in $C_2$.
    \item If $x_3 \leq 2x_2$, then $x_4 \leq 2x_2+x_4-x_3$. Thus, $x_{14} \leq x_{13}+x_{15}$ and we are in $C_3$.
\end{itemize} 
    
Now, we study the case where the tuple verifies $C_3$. Then, $x_{12}\geq x_{14} \geq x_{13}, x_{15}$. Thus,
\begin{itemize}
    \item If $2x_3 \geq 2x_2+x_4$, then $x_3 \geq 2x_2+x_4-x_3$. So $x_{14} \geq x_{13}+x_{15}$ and we are in $C_2$.
    \item If $2x_3 \leq 2x_2+x_4$, then $x_3 \leq 2x_2+x_4-x_3$. Therefore $x_{14} \leq x_{13}+x_{15}$ and we are in $C_3$ again.
\end{itemize} 

Finally, we consider the case $C_2$. Due to the relation between the initial conditions, $x_{11}\geq x_{14} \geq x_{13}$. Therefore,
\begin{itemize}
    \item If $x_1+x_2+x_4 \geq 2x_3$, then $x_1-x_3+x_2+x_4 \geq x_3$, which implies $x_{12} \geq x_{14}$. So, $x_{11} \geq x_{12} \geq x_{14}\geq x_{13}$ and we are in $C_1$. 
    \item If $x_1+x_2+x_4 \leq 2x_3$, then $x_1-x_3+x_2+x_4 \leq x_3$, thus $x_{12} \leq x_{14}$. So, $x_{11} \geq x_{14} \geq x_{12}\geq x_{13}$ and we are in $C_4$.
\end{itemize}

Once we have justified the diagram, it is highly relevant to assure that any tuple of initial conditions $(x_1,x_2,x_3,x_4)$ verify only one of the five classes $C_i$. Moreover, in principle, we have two possibilities for the way in which periodicity is achieved: either doing a cycle along the diagram starting and ending in the same case, or getting again the initial conditions in the passage between one case to another one in Diagram \ref{Diagrama}. 

\subsection{Controversial cases} \label{Sec_controversial} We assume that the initial conditions of the periodic sequence $(x_n)$ are $(x_1,x_2,x_3,x_4)$, with $x_1 = \max\{x_n:n\geq1\}$. Traveling by Diagram \ref{Diagrama}, after ten or eleven iterations, these initial conditions become into $(z_1,z_2,z_3,z_4)$ belonging to some case $C_i$. We will suppose that at some point of the movement in Diagram \ref{Diagrama} the new conditions $(z_1,z_2,z_3,z_4)$ in case $C_i$ goes to the tuple $(w_1,w_2,w_3,w_4)$ living in case $C_k$, and in the transition from $C_i$ to $C_k$, we get the initial conditions $(x_1,x_2,x_3,x_4)$. Notice that  $z_1 = x_1$. 

From now on, we will use $C_{i,j}$ to mean that a tuple $(z_1,z_2,z_3,z_4)$ verifies the case $C_i$ and that we achieve periodicity in the term $z_j$, $j=2,\ldots,11$ (if $i=1,3,4,5$), or $j=2,\ldots, 10$ (if $i=2$), that is, $z_j = x_1$, $z_{j+1} = x_2$, $z_{j+2} = x_3$ and $z_{j+3} = x_4$.

Now we will study each case. Due to the analogy between cases, we will only develop in detail the first one. In the others, we will only indicate the results. From now on, $x,y,z$ will be real numbers such that $x \geq y \geq z \geq 0$.

$\bullet$ Case $C_1$: By Proposition \ref{prop_diag} we know the following terms: $z_5 = z_2 - x_1 \leq 0;$  $z_6 = z_4 - z_2 \leq 0;$ $z_7 = z_4 - z_3;$ $z_8 = -z_3 \leq 0;$  $z_9 = z_4 - z_3 - z_2 + x_1;$ $z_{10} = x_1 - z_3;$ $z_{11} = x_1 - z_4;$ $z_{12} = x_1;$ $z_{13} = z_2 + z_3 - z_4;$ $z_{14} = z_3;$  $z_{15} = z_4$.

As $x_1 > 0$, we could reach periodicity in $z_2, z_3, z_4, z_7, z_9, z_{10}$ or $z_{11}$. Let us see the seven possible cases. 

$C_{1,2}: z_2 = x_1; \ z_3 = x_2; \ z_4 = x_3; \ z_2 - x_1 = x_4$ which implies that $(z_1,z_2,z_3,z_4) = (x_1,x_1,x_2,x_3)$. This is a controversial case, $(x,x,z,y)$. 

$C_{1,3}: z_3 = x_1; \ z_4 = x_2; \ z_5 = z_2 - x_1 = x_3 \leq 0; \ z_6 = z_4 - z_2 = x_4 \leq 0$. Notice that the two last equalities   imply $x_3 = 0$ and $x_4 = 0$. Then the initial conditions of the sequence are $(x_1, x_2, 0, 0)$ which is an $11$-cycle due to the monotonicity of the initial terms. 

$C_{1,4}: z_4 = x_1; \ z_2 - x_1 = x_2 \leq 0; \ z_4 - z_2 = x_3 \leq 0; \ z_4 - z_3 = x_4$. By the second and third equality we have that $x_2 = x_3 = 0$, so the initial conditions are $(x_1, 0, 0, x_4)$. Then we have a controversial case, $(x,0,0,y)$. 

$C_{1,7}: z_7 = z_4 - z_3 = x_1; \ -z_3 = x_2 \leq 0; \ z_4 - z_3 - z_2 + x_1 = x_3; \ x_1 - z_3 = x_4.  $ If we solve the system, we get that the initial conditions of the sequence are $(x_1, 0, x_3, x_1)$, but this tuple is an 8-cycle by Proposition \ref{P:periodo8}. 

$C_{1,9}: z_9 = z_4 - z_3 - z_2 + x_1 = x_1; \ x_1 - z_3 = x_2; \ x_1 - z_4 = x_3; \ x_1 = x_4$. These conditions imply that $(z_1,z_2,z_3,z_4) = (x_1, x_2 - x_3, x_1 - x_2, x_1 - x_3)$. Moreover, it verifies the case $C_1$, so we have that $z_1 \geq z_2 \geq z_4 \geq z_3$, which yields to $x_2 = x_1$ and then we are able to simplify the expression of the tuple into $(x_1, x_1-x_3, 0, x_1-x_3)$. In this sense we have reached another controversial case, that is, $(x,y,0,y)$. 

$C_{1,10}: z_{10} = x_1 - z_3 = x_1; \ x_1-z_4 = x_2; \ x_1 = x_3; \ z_2+z_3-z_4 = x_4$. That yields to $(z_1,z_2,z_3,z_4) = (x_1,x_1+x_4-x_2,0,x_1-x_2)$, which is a controversial case of the type $(x,y,0,z)$. 

$C_{1,11}: z_{11} = x_1 - z_4 = x_1; \ x_1 = x_2; \ z_2+z_3-z_4 = x_3; \ z_3 = x_4$. Notice that by the second and third equation we get $z_4 = 0$ and $z_2= x_3 - x_4$. Therefore, the tuple is of the form $(x_1, x_3-x_4,x_4,0)$, but as they verify the case $C_1$ we have that $x_4 = 0$. So, the four terms are monotonic and we can assure that the sequence is an 11-cycle. 

$\bullet$ Case $C_2:$ We compute the following terms of the tuple $(x_1,z_2,z_3,z_4)$. By Proposition \ref{prop_diag}: $z_5 = z_3 - x_1 \leq 0$, $z_6 =z_3 - z_2$, $z_7 = -z_2 \leq 0$, $z_8 = z_3 - z_2 - z_4$, $z_9 = x_1 - z_2$, $z_{10} =x_1 - z_3$, $z_{11} = x_1$, $z_{12} = x_1 - z_3 + z_2+z_4$, $z_{13} = z_2 $ and $z_{14} = z_3$. Now, if we consider the positive terms, the cycle can take place in $z_2, z_3, z_4,$ $z_6, z_8, z_9, z_{10}$. We indicate what happens in each case.

$C_{2,2}:$ we have an eleven cycle by monotonicity of the initial conditions. 

$C_{2,3}:$ we get the controversial case $(x,0,z,y)$.

$C_{2,4}:$ the initial conditions are $(x_1,0,x_1,0)$ and we have an eight cycle. 

$C_{2,6}:$ the initial conditions are $(x_1,0,x_3,x_1)$ and we have an eight cycle. 

$C_{2,8}:$ the initial conditions are $(x_1, x_1, 0, x_1)$, which is an eight cycle. 

$C_{2,9}:$ we get the controversial case $(x,0,y,z)$. 

$C_{2,10}:$ we have an eleven cycle by monotonicity of the initial conditions.

$\bullet$ Case $C_3$: We know the following terms by Proposition \ref{prop_diag}:  $z_5 = z_3 - x_1 \leq 0$, $z_6 =z_3 - z_2$, $z_7 = z_4 - z_3 \leq 0$, $z_8 = z_3 - z_2 - z_4\leq 0$, $z_9 = x_1 - z_2$, $z_{10} =x_1 - z_3$, $z_{11} = x_1 - z_4 + z_3 - z_2$ , $z_{12} = x_1$, $z_{13} = z_2 $, $z_{14} = z_3$ and $z_{15} = z_4 + z_2 - z_3$. As in the previous cases, we can only have the cycle in the positive terms, that is, $z_2, z_3, z_4, z_6, z_9, z_{10}, z_{11}$. Let us study the different possibilities. 

$C_{3,2}:$ the initial conditions are monotonic, so we have an eleven cycle. 

$C_{3,3}:$ the initial conditions are the controversial case $(x,y,0,z)$. 

$C_{3,4}:$ we have a controversial case of the form $(x,0,y,0)$. 

$C_{3,6}:$ the initial conditions are $(x_1,0,0,x_1)$ which is an eight cycle.

$C_{3,9}:$ we arise to the controversial case $(x,0,z,y)$. 

$C_{3,10}:$ we have monotonicity, so it is an $11$-cycle. 

$C_{3,11}:$ we get the controversial case $(x,x,z,y)$.

$\bullet$ Case $C_4$: By Proposition \ref{prop_diag}, the sequence evolves as follows: $z_5 = z_4 - x_1 \leq 0$, $z_6 =z_4 - z_2$, $z_7 = z_4 - z_3$, $z_8 = -z_3 \leq 0 $, $z_9 = x_1 - z_3$, $z_{10} = x_1 + z_2 - z_3 - z_4$, $z_{11} = x_1 - z_4$, $z_{12} = x_1$, $z_{13} = z_3 $, $z_{14} = z_4 + z_3 - z_2$ and $z_{15} = z_4$. Now, we must study when the positive terms, $z_2, z_3, z_4, z_6, z_7, z_9, z_{10},$ $z_{11}$ are equal to $x_1$ and we have the cycle. 

$C_{4,2}:$ yields to a controversial case of the form $(x,x,y,x)$. 

$C_{4,3}:$ $(z_1,z_2,z_3,z_4)=(x_1,x_1-x_4,0,x_1)$, that is an eight cycle.

$C_{4,4}:$ we arise to the controversial case $(x,0,z,y)$.

$C_{4,6}:$ the tuple $(z_1,z_2,z_3,z_4) = (x_1,0,0,x_1)$, which is an eight cycle. 
$C_{4,7}:$ the tuple $(z_1,z_2,z_3,z_4)=(x_1,x_4,0,x_1)$, which is an eight cycle. 
$C_{4,9}:$ we get a controversial case of the type $(x,z,0,y)$.

$C_{4,10}:$ we get the controversial case $(x,y,0,y)$. 

$C_{4,11}:$ we have monotonicity in the tuple, so it is an eleven cycle.

$\bullet$ Case $C_5$: We proceed analogously to the previous cases:
$z_5 = z_4 - x_1 \leq 0$, $z_6 =z_4 - z_2$, $z_7 = z_4 - z_3$, $z_8 = -z_2 \leq 0$, $z_9 = x_1 - z_2$, $z_{10} = x_1- z_4$, $z_{11} = x_1 - z_2 - z_4 + z_3$, $z_{12} = x_1$, $z_{13} = z_2 $, $z_{14} = z_4$ and $z_{15} = z_2 + z_4 - z_3$. Notice that the cycle can occur in the terms $z_2,z_3,z_4,z_6,z_7,z_9,z_{10}$ and $z_{11}$, since they are non-negative. Let us study each case. 

$C_{5,2}:$ we get the controversial case $(x,y,x,0)$.

$C_{5,3}:$ we get the controversial case $(x,x,0,y)$. 

$C_{5,4}:$ the initial conditions are of the form $(x,0,y,z)$. 

$C_{5,6}:$ $(z_1,z_2,z_3,z_4) = (x_1,0,x_1-x_2,x_1)$ and we have an eight cycle.

$C_{5,7}:$ $(z_1,z_2,z_3,z_4) = (x_1,0,0,x_1)$ is an eight cycle.

$C_{5,9}:$ we have a controversial case $(x,0,y,z)$.

$C_{5,10}:$ we achieve an eleven cycle due to monotonicity in the initial terms.

$C_{5,11}:$ yields to the controversial case $(x,x,0,y)$.

After that analysis, we collect the controversial cases in the following table (remember that $x \geq y \geq z \geq 0$):
\begin{table}[h]
\begin{tabular}{ccccc}
$(x,x,z,y)$ & $(x,0,0,y)$ & $(x,y,0,y)$  & $(x,y,0,z)$ & $(x,x,0,y)$ \\
$(x,0,y,0)$ & $(x,x,y,x)$  & $(x,0,z,y)$ & $(x,z,0,y)$ & $(x,0,y,z)$ 
\end{tabular}
\end{table}

Nevertheless, we can reduce some of the previous tuples into equivalent classes (see Definition \ref{Def_equivalence}).

\begin{proposition} \label{prop_controversial}
Let $x,y,z$ be real numbers such that $x \geq y \geq z \geq 0$. Then we have the following relations: $$(x,x,z,y) \sim (x,z,y,0); \ \ \  (x,x,y,x) \sim (x,y,x,0); \ \ \ (x,x,0,y) \sim (x,0,y,0).$$
\end{proposition}

\begin{proof}
We only have to compute each tuple under Equation (\ref{Eq:G4}). For instance, if $x_1=x$, $x_2=x$, $x_3=z$ and $x_4 =y$, then $x_5 = 0$ and  $(x_1,x_2,x_3,x_4) \sim (x_2,x_3,x_4,x_5)$, or equivalently, $(x,x,z,y) \sim (x,z,y,0)$. The rest of cases can be sum up as follows: $(x,x,y,x,0); (x,x,0,y,0).$
\end{proof}

Therefore, we only have to consider the following nine cases:
$$(x,0,0,y); \ \ (x,y,0,y); \ \ (x,0,y,0); \ \ (x,y,x,0); $$
$$(x,z,y,0); \ \ (x,y,0,z); \ \ (x,z,0,y); \ \ (x,0,y,z); \ \ (x,0,z,y). $$ 
Notice that the first four cases formed by two positive variables, $x,y$, are included in the cases with three non negative variables, $x,y,z$. For instance, $(x,0,0,y)$ is of the form $(x,z,0,y)$ with $z=0$; $(x,y,0,y)$ is of the form $(x,z,0,y)$ with $z=y$; $(x,0,y,0)$ is like the initial conditions $(x,0,y,z)$ with $z=0$; and $(x,y,x,0)$ behaves as $(x,z,y,0)$ with $x=y$. Therefore, our controversial cases are reduced to:
\begin{equation} \label{Cases_equiv}
     (x,y,0,z); \ \ (x,z,0,y); \ \ (x,0,y,z); \ \ (x,0,z,y); \ \ (x,z,y,0).
\end{equation}

Additionally, the initial conditions $(x,z,0,y)$ and $(x,0,y,z)$ generate the same sequence, since \[\textsf{x, z, 0, y,} \ y-x, y-z, y, 0, x, x-y+z, x-y, x, 0,\] \[y-z, y, y-x, y, z, 0, x, x-y, x-z, \ \textsf{x, 0, y, z}.\] Thus, we get $(x,z,0,y) \sim (x,0,y,z).$

Also, if we compute the sequence generated by Equation (\ref{Eq:G4}) from the initial conditions $(x,0,z,y)$, we get $x, 0, z, y, y-x, y, y-z, 0, x, x-y, x-y+z, x, 0, y, y-z.$ Now then, $(x,0,z,y) \sim (x,0,y,y-z)$, and as $y\geq y-z,$ the new tuple $(x,0,y,y-z)$ has the form $(x,0,y,z')$, with $x\geq y \geq z' \geq 0$. So, if we know the behaviour of the tuple $(x,0,y,z')$ under Equation (\ref{Eq:G4}), we will know the evolution of $(x,0,z,y)$ too. Moreover, something similar happens with the initial conditions $(x,y,0,z)$, since after $11$ iterations, it becomes into $(x,y-z,0,z)$; if $y-z \leq z$, then $(x,y,0,z) \sim (x,z',0,y')$; on the contrary, repeating the process we arrive to $(x,y-2z,0,z)$ after 11 iterations. It is easy to check by induction that if we have the tuple $(x,y-mz,0,z)$ with $m\geq 0$ and $y-mz > z$, after 11 iterations we arrive to $(x, y - (m+1)z,0,z)$; thus at some point the second term will be less than the fourth and we will get the previous equivalence. To sum up, we have the equivalence $(x,y,0,z) \sim (x,z',0,y')$ for some $x \geq y'\geq z'$ and if we know the behaviour of $(x,z',0,y')$ under Equation (\ref{Eq:G4}), we will know the evolution of $(x,y,0,z)$.

In conclusion, in view of (\ref{Cases_equiv}) and the previous considerations, it is only necessary to analyse $(x,y,0,z)$ and $(x,z,y,0)$, since from their behaviour we can deduce how all the controversial cases will evolve.

\subsection{Some specific initial conditions}\label{Sec_specific}  Now we are able to analyse in detail what occurs when we compute the initial conditions $(x,y,0,z)$ and $(x,z,y,0)$ under Equation (\ref{Eq:G4}).

\subsubsection{\textbf{Case} $\textsf{(x,y,0,z)}$} We begin by studying the evolution of $(x,y,0,z)$ under Equation~(\ref{Eq:G4}), with $x \geq y \geq z \geq 0$. Notice that if $z=0$, the initial conditions generate an 11-cycle due to monotonicity, see Proposition \ref{Prop_monotonic}. On the other hand, if $x=z\geq 0$, then the initial conditions can generate an 8-cycle or the equilibrium solution. Thus, we will assume $x \geq y \geq z > 0$ and $x > z$ to avoid trivial cases (periods $1, 8$ or $11$). 

\begin{lemma} \label{Lemma_xy0z}
Let us consider the tuple $(x,tx+y-sz,0,z)$, with $x\geq tx+y-sz \geq 0$, $t\geq 0$ and $s\geq 0$. Then,
\begin{itemize}
    \item[(a)] If $tx + y - sz > z$, then after 11 iterations the tuple becomes into $(x, tx + y - (s+1)z, 0, z)$.
    \item[(b)] If $tx + y - sz \leq z$, then after 32 iterations the tuple becomes into $(x, (t+1)x + y - (s+1)z, 0, z)$.
\end{itemize}
\end{lemma}

\begin{proof}
Assume $tx + y - sz > z$ and compute $(x,tx+y-sz,0,z)$ under (\ref{Eq:G4}): $$x, tx + y - sz, 0, z, (t-1)x + y - sz, (s+1)z - tx - y, z, 0,$$ $$(s+1)z - (t-1)x - y, x, x-z, x, tx + y - (s+1)z, 0, z.$$

On the other hand, let us assume that $tx + y - sz \leq z$. Then, $$x, tx+y-sz, 0, z, z-x, -tx - y + (s+1)z, z, 0, x,$$ $$(1+t)x + y - (s+1)z, x-z, x, 0, -tx - y + (s+1)z, z, z-x, z,$$ $$ tx+ y - sz, 0, x, x-z, (1-t)x - y + sz, x, 0, z, tx + y - sz, z-x,$$ $$z, 0, -tx - y + (s+1)z, x, x-z, x, (1+t)x + y - (s+1)z, 0, z. $$
\end{proof}

\begin{remark} \label{Remark_xy0z}
Notice that the initial conditions, $(x,y,0,z)$, are of the type $(x,tx + y - sz, 0,z)$ with $t=0$ and $s=0$. Thus, its evolution under Equation (\ref{Eq:G4}) is perfectly described by Lemma \ref{Lemma_xy0z}.

Now, according to Lemma \ref{Lemma_xy0z}, we will analyse the possibility of achieving periodicity in the middle of the process, that is, if at some point of its evolution, it is possible to recover the initial conditions $(x,y,0,z)$. For that, it is necessary to check the block of iterations in cases (a) and (b) of Lemma \ref{Lemma_xy0z}.

Firstly, in the case (a), reasoning similarly to the study of controversial cases in Subsection \ref{Sec_controversial}, it is easily seen that periodicity cannot hold in the middle of the process.

However, in case (b), we could obtain periodicity in the twenty second iteration, which implies $$(1-t)x - y + sz = x; \ \ \ x = y; \ \ \ 0 = 0; \ \ \ z = z. $$

Then, the initial conditions are $(x,x,0,z)$. Moreover, $(1-t)x - x + sz = x$ must hold and we get $sz = (t+1)x$. Therefore, when analysing the period of the sequence, this case must be study apart. $\hfill\Box$
\end{remark}

Now, we will study the periodic character of the sequence generated by the initial conditions $(x,y,0,z)$.

\begin{lemma} \label{Lemma2}
Assume that the initial conditions $(x,y,0,z)$ generate a periodic sequence under Equation (\ref{Eq:G4}), with $x\geq y \geq z > 0$, $x>z$. Then $\frac{z}{x} \in \mathbb{Q}$.
\end{lemma}

\begin{proof}
Firstly, we divide each component by $x$, so we get $(1, \frac{y}{x}, 0, \frac{z}{x})$. If we apply Lemma \ref{Lemma_xy0z} to $(x,y,0,z)$, then for some $t, s \geq 1$ we will have periodicity if $tx+y-sz = y$ (in fact, there are an infinity of values $t,s$ holding the equality), which implies that $\frac{z}{x} = \frac{t}{s} \in \mathbb{Q}$.
\end{proof}

The following result shows us that the restriction $\frac{z}{x} \in \mathbb{Q}$ is not only necessary, but also sufficient. According to Lemma \ref{Lemma2} and Proposition \ref{Prop_scalar}, we can simplify the tuple $(1, \frac{y}{x}, 0, \frac{z}{x})$ as $(1, \frac{y}{x}, 0, \frac{q}{p})$, where $\gcd(p,q) = 1$. Finally, multiplying by $p$ we get $(p, \bar{y}, 0, q)$, where $\bar{y} = p \cdot \frac{y}{x}$. Notice that $\bar{y}$ does not need to be rational. 

\begin{proposition} \label{N_controversial1}
Let $p,q\in\mathbb{Z}$, $\bar{y} \in \mathbb{R}$. Assume $p\geq \bar{y} \geq q>0$ and $\gcd(p,q) = 1$. The initial conditions $(p, \bar{y}, 0, q)$ generate under Equation (\ref{Eq:G4}) a periodic sequence. Also, the period is given by $N=(p+q)\cdot 11 + q\cdot 10$.
\end{proposition}

\begin{proof}
By Lemma \ref{Lemma_xy0z} (with $x=p$, $z = q$ and $t=0=s$), we know that the second term of the initial conditions is the only one that varies after 11 or 32 iterations. In fact, if we denote by $(z_1,z_2,z_3,z_4)$ the tuple after each block of 11 or 32 iterations (observe that $z_1=p$, $z_2 \geq 0$, $z_3=0$ and $z_4 = q$), then we claim that $z_2 = \{\bar{y}\} + j$ for some $j \in \{0,1,\ldots, p-1, p\}$, where $\{\cdot\}$ denotes the fractional part of a number. Indeed, in the first time, $z_2=\bar{y}=\{\bar{y}\} + [\bar{y}]$, thus the property is true for $j=[\bar{y}]\in \{0,1,\ldots, p\}$ (notice that $j=p$ is possible when $\{\bar{y}\} = 0$ and $p=\bar{y}$). Next, since $\bar{y} \geq q$, if $\bar{y} > q$, by Lemma \ref{Lemma_xy0z}, $\bar{y}$ goes to $z_2 = \bar{y} - q < p$. On the other hand, if $\bar{y} = q$, then $z_2 = p + \bar{y} - q = p$ after 32 iterations. In both cases, $z_2 = \{\bar{y}\} + j$ for appropriate $0\leq j \leq p$. By induction it is immediate to check that after each block of the 11 or 32 iterations stipulated in Lemma \ref{Lemma_xy0z}, we find that the corresponding $z_2$ has the form searched, which ends the claim.  

Therefore, as $z_2$ can only take a finite number of values, and we know by Lemma \ref{Lemma_xy0z} that eventual periodicity implies periodicity, we can assure that the value $z_2 = \bar{y}$ will be repeated and the sequence will be periodic. Now, we will determine the period of the sequence.

According to Remark \ref{Remark_xy0z} we need to distinguish the cases $p>\bar{y}$ and $p = \bar{y}$. In the first case periodicity holds after a finite amount of complete blocks of $11$ or $32$ iterations. On the other hand, when $p = \bar{y}$, periodicity holds in the middle of a block of $32$ iterations.

Firstly, we assume that $p>\bar{y}$. Realize that after applying $r$ times a block of 32 iterations and $m$ times a block of 11 iterations, by Lemma \ref{Lemma_xy0z}, the second term will be $z_2= \bar{y} + r\cdot p - (r+m) \cdot q$.  

Let us denote by $\alpha$ and $\beta$ the number of blocks of 32 and 11 iterations, respectively, needed to achieve periodicity. Thus, $\alpha\cdot p + \bar{y} - (\alpha + \beta) \cdot q = \bar{y}$, which implies $\alpha \cdot p = (\alpha + \beta) \cdot q$. Observe that gcd$(p,q)=1$ yields to $\alpha = \dot{q}$ and $\alpha + \beta = \dot{p}$. Nevertheless, as the integer part of the second term, $[z_2]$, varies in the set $\{0,1,\ldots,p-1,p\}$, then $\alpha + \beta$ (which represents the total amount of blocks of 11 and 32 iterations needed to achieve periodicity) must be less or equal to $p+1$. So, if $\alpha+\beta =\dot{p}$ and $\alpha + \beta \leq p+1$, it holds $\alpha + \beta = p$ (observe that $p \geq 2$). Then, $\alpha \cdot p = p \cdot q$ and it follows $\alpha = q$, since gcd$(p,q)=1$. 

Definitely, the sequence will be determined by $(p-q)$ blocks of $11$ iterations and $q$ blocks of $32$. Therefore, the period of the sequence is $N = (p-q)\cdot 11 + q \cdot 32$, which is equivalent to $$N = (p+q)\cdot 11 + q \cdot 10.$$

Secondly, assume that $p = \bar{y}$. By Remark \ref{Remark_xy0z}, this means that after applying $r=t$ times a block of 32 iterations and $m = s-t$ times a block of 11 iterations, periodicity will hold in the middle of the following block of 32 iterations (thus, we assume that we have not recover the initial conditions before those $s$ blocks); moreover, $sq = (t+1)p$, which implies $\frac{q}{p} = \frac{t+1}{s}$. As gcd$(q,p)=1$, $\frac{q}{p}$ is irreducible and we deduce that $t+1\geq q$, $s\geq p$.

If $s > p$, then $t+1 > q$ must hold.
This means that there are $s$ tuples, with $s = s'p, s'>1, s\geq 2p$, that are of the form $(p, \widetilde{t}p + \bar{y} - \widetilde{s}q, 0, q)$, but we know that every term $\widetilde{t}p + \bar{y} - \widetilde{s}q$ is of the type $\bar{y}+j$, with $j\in\{0,1,\ldots,p\}$.
 
As $s\geq 2p$ and $2p > p+1$ if $p\geq 2$ ($p=1$ yields to $(1,1,0,1)$ that is an $8$-cycle; also $p>q$), there exist $\widetilde{s}, \widetilde{t}$, with $\widetilde{s}\in[1,p], \widetilde{t}\leq \widetilde{s}$, such that $\bar{y} + \widetilde{t}p - \widetilde{s}q = \bar{y}$. Since $\widetilde{s}<s$ we obtain a contradiction, as we have supposed that periodicity has not been completed before applying $s$ blocks of iterations.

So, $s=p, \ t+1=q$, and consequently the period is
\begin{eqnarray*}
N &=& 32\cdot (q-1) + 11\cdot (p-q+1) + 21 = (2\cdot 11 + 10) \cdot q + 11\cdot (p-q) \\
&=& (p+q)\cdot 11 + q\cdot 10.
\end{eqnarray*}
\end{proof}

\begin{remark} \label{Remark_controv1}
Notice that the coefficients of the decomposition, that is, $a=q$ and $b= p + q$ verify the conditions $b \geq 2a + 1$ and $\mathrm{gcd}(a,b) = 1$. Moreover, every number of the form $\widetilde{N} = 10\cdot a + 11 \cdot b$ with $b \geq 2a + 1$ and $\mathrm{gcd}(a,b) = 1$ can be written as $\widetilde{N} = 10\cdot q + 11 \cdot (p+q)$ by taking $q=a$ and $p=b-a$, with $\gcd(p,q) = 1$ and $p\geq 2q+1$ as can be easily checked.
\end{remark}

\begin{corollary} \label{Cor_xy0z}
Given the initial conditions $(x,y,0,z)$ with $x \geq y \geq z > 0$ and $x>z$, they generate a periodic sequence under Equation (\ref{Eq:G4}) if and only if $\frac{z}{x}\in \mathbb{Q}$. In that case, the period is $N= 10\cdot a + 11\cdot b$, for some $a,b \in \mathbb{N}$ where $\gcd(a,b) = 1$, $b\geq 2a + 1$. 
\end{corollary}

\subsubsection{\textbf{Case} $\textsf{(x,z,y,0)}$} \label{Sec_xzy0}  

As in the previous case, we assume $x\geq y \geq z \geq 0$. If $y=z$, the initial conditions provide an 11-cycle. If $z=0$, $(x,0,y,0)$ is equivalent to $(x,x-y,0,y)$ if $x-y\geq y$ (it is only necessary to compute a block of ten iterations to see it); and it is equivalent to $(x,0,-x+2y,y)$ if $x-y<y$ (it follows by computing a block of 21 iterations). So, in the first case, we arrive to a set of initial conditions having the structure of $(x,y',0,z')$; and in the second case, we get to $(x,0,z',y')$, with $x\geq y' \geq z' \geq 0$. Both initial conditions have been already analysed in the previous subsection. Therefore, from now on, we will assume that $$y>z>0.$$

On the other hand, observe that after ten iterations the tuple becomes into $(x,x-(y-z),z,y)$. Indeed: \[x, z, y, 0, y-x, y-z, -z, y-z, x-z, x-y, x, x - (y-z), z, y.\]

\begin{lemma} \label{prop_xzy1}
Let us consider the tuple $(x, tx - s(y-z), z, y)$ with $x \geq y > z > 0$, $x \geq tx - s(y-z)\geq 0$ and $t,s$ nonnegative integers. Then,
\begin{itemize}
    \item[(a)] If $tx - s(y-z) \geq y$, then after 11 iterations the tuple becomes into $(x, tx-(s+1)(y-z), z, y)$.
    \item[(b)]  If $tx - s(y-z) < y$ and $tx - (s-1)(y-z) \geq y$, then after 22 iterations the tuple becomes into $(x, z, y, tx-s(y-z))$.
\end{itemize}
\end{lemma}

\begin{proof}
First, assume $tx - s(y-z) \geq y$. Then by Equation (\ref{Eq:G4}):
$$x, tx - s(y-z), z, y, (t-1)x - s(y-z), (s+1)y - tx - sz, y-z, -z,$$ $$ (1-t)x + (s+1)(y-z), x-z, x-y, x, tx-(s+1)(y-z), z, y. $$
On the other hand, if $tx - s(y-z) < y$ and $tx - (s-1)(y-z) \geq y$, then: $$x, tx - s(y-z), z, y, y-x, (s+1)y - tx - sz, y-z, -z, x-z,$$ $$ (t+1)x + (s-1)z - (s+1)y, x-y, x, z, -tx+(s+1)y - (s-1)z, y, y-x,$$ $$y-z, tx-sy+(s-1)z, -z, x-z, x-y, (1-t)x+sy-sz, x,z,y,tx-s(y-z). $$
\end{proof}

\begin{lemma} \label{prop_xzy2}
Let us consider the tuple $(x,z,y,tx-s(y-z))$ with $x \geq y > z > 0$, $y> tx-s(y-z) \geq 0$ and $t,s$ nonnegative integers. Then, 
\begin{itemize}
    \item[(a)] If $tx-(s+1)(y-z) \geq 0$, then after 11 iterations the tuple becomes into $(x,z,y,tx-(s+1)(y-z))$.
    \item[(b)] If $tx-(s+1)(y-z) < 0$, then after 10 iterations the tuple becomes into $(x,(t+1)x-(s+1)(y-z),z,y)$.
\end{itemize}
\end{lemma}

\begin{proof}
Assume $tx-(s+1)(y-z) \geq 0$. Therefore, by Equation (\ref{Eq:G4}): $$x, z, y, tx - s(y-z), y-x, y-z, tx - (s+1)y+sz, -tx + (s+1)(y-z), x-z,$$ $$x-y, (1-t)x + (s+1)y - (s+1)z, x, z, y, tx - (s+1)(y-z). $$
On the contrary, if $tx-(s+1)(y-z) < 0$, then the tuple evolves as follows: $$x, z, y, tx-s(y-z), y-x, y-z, -z, -tx + (s+1)(y-z), x-z, x-y,$$ $$x, (t+1)x-(s+1)(y-z),z,y. $$
\end{proof}

\begin{remark} \label{Remark3}
Observe that if four consecutive terms verify the conditions of Lemma \ref{prop_xzy1}(b), after 22 iterations, the new terms will satisfy the conditions of Lemma \ref{prop_xzy2}. Furthermore, if four consecutive terms verify the conditions of Lemma \ref{prop_xzy2}(b), then, after 10 iterations, the new tuple will satisfy the conditions of Lemma \ref{prop_xzy1}. Therefore, Lemma \ref{prop_xzy1} and Lemma \ref{prop_xzy2} jointly describe how the tuple $(x, x - (y-z),z,y)$, which appears after iterating 10 times the initial conditions $(x,z,y,0)$, evolves under Equation (\ref{Eq:G4}). $\hfill\Box$
\end{remark}

\begin{remark} \label{Remark_xzy04}
Now according to Lemma \ref{prop_xzy1} and Lemma \ref{prop_xzy2} we will analyse the possibility of achieving periodicity in the middle of the process, that is, if at some point of its evolution, it is possible to recover the initial conditions $(x,x-(y-z),z,y)$ (observe that if the initial conditions $(x,x-(y-z),z,y)$ generate a periodic sequence, then, by Proposition \ref{Prop_eventually}, the same applies for the initial conditions $(x,z,y,0)$ and the period of the sequence is obviously the same). For that, it is necessary to check the blocks of iterations in cases (a) and (b) of Lemma \ref{prop_xzy1} and Lemma \ref{prop_xzy2}.

Firstly, in Lemma \ref{prop_xzy1}(b), we could obtain periodicity in the eleventh iteration, which implies
$$ x = x; \ \ \ \ z = x - (y-z); \ \ \ \ -tx+(s+1)x-(s-1)z = z; \ \ \ \ y = y. $$

This yields to $x=y$, so the initial conditions are $(x,z,z,x)$. Now, if we iterate this tuple under Equation (\ref{Eq:G4}), we get $x, z, z, x, 0, x-z, x-z$, so $(x,z,z,x) \sim (x,0,x-z,x-z)$. Observe that the tuple $(x,0,x-z,x-z)$ behaves as the initial conditions $(x,0,y',y')$ with $x>y'>0$, but this tuple is a particular case of the controversial one $(x,0,\widetilde{y},\widetilde{z})$ with $\widetilde{y}=y'=\widetilde{z}$. Moreover, in Subsection \ref{Sec_controversial}, we have seen that, due to the equivalence relations, this case is covered by the analysis of the initial conditions $(x,\hat{y},0,\hat{z})$ with $x\geq \hat{y} \geq \hat{z}\geq 0$ and $x > \hat{z}$.

On the other hand, it is easy to check, proceeding as in the study of the controversial cases (see Subsection \ref{Sec_controversial}), that periodicity cannot hold in the middle of the process of the block of 11 iterations of Lemma \ref{prop_xzy1}(a), nor in the middle of blocks of 11 and 10 iterations of Lemma \ref{prop_xzy2}. $\hfill\Box$
\end{remark}

By Lemma \ref{prop_xzy1}, Lemma \ref{prop_xzy2} and Remark \ref{Remark3}, we know how a tuple of the form $(x,z, y,0)$ with $x\geq y > z > 0$ evolves under Equation (\ref{Eq:G4}). Now, we will determine the period of the sequence. For that, we will distinguish two cases depending on the fact whether the initial terms are rational numbers or not.

\textbf{Case A:} If the initial conditions are rational, we can multiply them by the least common multiple of their denominators and the terms will be integers. Thus, we can assume without loss of generality, by Proposition \ref{Prop_scalar}, that $x, y, z \in \mathbb{N}$. 

Let us call $d = \gcd(y-z,x)$. After the first ten iterations, we have seen that the original tuple becomes into $(x, x -(y-z), z, y)$. If we divide each term by $d$, we get $\left( \frac{x}{d}, \frac{x-(y-z)}{d}, \frac{z}{d}, \frac{y}{d} \right)$, or equivalently $(p,q, z', y')$, where $y' = \frac{y}{d}$, $z' = \frac{z}{d}$ and $\gcd(p,q) = 1$. Notice that $p$ is the greatest term of the tuple and $p$ and $q$ are integer numbers.

\begin{proposition} \label{N_controversial2}
Under the previous considerations, the initial conditions $(p,q, z', y')$, with $\gcd(p,q) = 1$ and $q=p-(y'-z')$, generate under Equation (\ref{Eq:G4}) a periodic sequence. Furthermore, the period is given by $N=(2p-q)\cdot 11 + (p-q)\cdot 10$.
\end{proposition}

\begin{proof}
Firstly, as $q=p-(y'-z')$, the initial conditions $(p,q,z',y')$ satisfy the conditions of Lemma \ref{prop_xzy1} with $x=p$, $z=z'$, $y=y'$ and $t=1=s$. Also, by Remark \ref{Remark3}, its evolution under Equation (\ref{Eq:G4}) will be described by Lemmas \ref{prop_xzy1} and \ref{prop_xzy2}. 

Let us denote by $b_0^{(2)} := q$, where the superscript $(2)$ represents that the term $q$ is placed in the second position in the initial tuple. Now, depending on the requisites over the initial conditions, we will apply Lemma \ref{prop_xzy1}(a) or (b). In any case, we obtain, after the corresponding block of 11 or 22 iterations, a new tuple $(p,b_1^{(2)}, z',y')$ or $(p,z',y', b_1^{(4)})$, with $b_1^{(2)} = p - 2(y'-z') = q - (y'-z')$ and $b_1^{(4)} = q$, where we use $(2)$ or $(4)$ according to the fact that the term which has varied is in the second or fourth position. So, following the process (by applying Lemma \ref{prop_xzy1} or \ref{prop_xzy2} and taking into account the only term that varies its value), we can define a sequence $b_j^{(i_j)}$, with $j\geq 1, i_j\in \{2,4\}$. If $i_j = 2$, the new tuple is $(p,b_j^{(2)}, z',y')$; otherwise, we obtain $(p,z',y', b_j^{(4)})$. 

Once the sequence $(b_j^{(i_j)})$ is defined, it must be highlighted that $b_j^{(i_j)} \in \{0^{(2)}, 0^{(4)},1^{(2)},1^{(4)},\ldots,p^{(2)},p^{(4)}\} =: \mathcal{V}$. This is a consequence from the fact that $y'-z' = \frac{y-z}{d}$ is an integer number and the conditions of Lemmas \ref{prop_xzy1} and \ref{prop_xzy2} (it is clear that $b_0^{(2)}$ verifies the claim and now, by induction, it is easily seen that any $b_j^{(i_j)}$ will hold the property too; for instance, if we are in Lemma \ref{prop_xzy1}(a), then $b_j^{(2)} = tp - (s+1)(y'-z')$ for some integers $t,s\geq 1$ and $b_j^{(2)} > p$ would imply $tp - s(y'-z') > p + y'-z' > p$ which contradicts the initial hypothesis of Lemma \ref{prop_xzy1}; the remaining cases are analogous).

Next, we claim that $b_j^{(i_j)} \notin\{0^{(2)}, p^{(4)}\}$. This means that we cannot assign the value $0$ in the second position of the corresponding tuple nor $p$ in the fourth position. Again, this is a direct consequence from the hypothesis of Lemmas \ref{prop_xzy1} and \ref{prop_xzy2}. For instance, if in Lemma \ref{prop_xzy2}(b) we would obtain the tuple $(p, (t+1)p - (s+1)(y'-z'), z', y') = (p,0,z',y')$, then $tp - s(y'-z') = -p + y'-z' < 0$, which contradicts that $tp - s(y'-z') \geq 0$ in the initial hypothesis; the remaining cases are similar.   

Therefore, $b_j^{(i_j)} \in \mathcal{V} \setminus \{0^{(2)}, p^{(4)}\}$. So, we can only have a finite number of tuples $(p, b_j^{(2)}, z', y')$ or $(p,z',y',b_j^{(4)})$. By Lemmas \ref{prop_xzy1} and \ref{prop_xzy2}, we do a circuit with these tuples, thus we can assure that some tuple returns to itself and by Lemma \ref{Prop_eventually} all the sequence is periodic. Moreover, as $\mathrm{Card}(\mathcal{V} \setminus \{0^{(2)}, p^{(4)}\}) = 2p$, where $\mathrm{Card}$ denotes the cardinality of a set, periodicity will be attained by applying Lemmas \ref{prop_xzy1} and \ref{prop_xzy2} at most $2p$ times. 

Now we will determine the period of the sequence. According to Remark \ref{Remark_xzy04}, we know that is only necessary to assume that periodicity holds after a complete sequence of blocks described in Lemmas \ref{prop_xzy1} and \ref{prop_xzy2}. Denote by: $\alpha$ the number of blocks of 11 iterations described by Lemma \ref{prop_xzy1}(a); $\beta$ the amount of blocks of 22 iterations described by Lemma \ref{prop_xzy1}(b); $\gamma$ the number of blocks of 11 iterations described by Lemma \ref{prop_xzy2}(a); and $\delta$ the amount of blocks of 10 iterations described by Lemma \ref{prop_xzy2}(b), needed to achieve periodicity. 

Firstly, we claim that $\beta = \delta$. Indeed, the block of 22 iterations described by Lemma \ref{prop_xzy1}(b) changes the position of the term $b_j^{(2)}$ from the second to the fourth, but it does not vary its value; in the other hand, the block of 10 iterations described by Lemma \ref{prop_xzy2}(b), changes again the position of $b_j^{(4)}$ from the fourth to the second, so necessarily we must have the same number of blocks of each type, since every change of position from the second to the fourth place must be balanced with another change from the fourth to the second. Furthermore, $\beta \geq 1$ because the initial tuple $(p,q,z',y')$ satisfies the hypothesis of Lemma \ref{prop_xzy1} and at some point we will have $p - (m+1)(y'-z') < y'$ for some positive integer $m$ and case (b) will be applied. 

Secondly, according to the concatenation of blocks given by Remark \ref{Remark3}, periodicity will hold when 
\begin{equation} \label{pepe}
    (\delta+1)p - (\alpha+\gamma+\delta + 1)(y'-z') = q = p-(y'-z').
\end{equation}
Indeed, observe that the coefficient of $p$ is increased by one unit only if we apply the block of ten iterations of Lemma \ref{prop_xzy2}(b); on the other hand, the coefficient of $y'-z'$ remains unchanged when we apply the block of twenty two iterations of Lemma \ref{prop_xzy1}(b), and it is increased by one otherwise.

By (\ref{pepe}), $\delta p = (\alpha + \gamma + \delta)(y'-z')$. Observe that gcd$(p,q)=1$ yields to gcd$(p,y'-z')=1$, since $q=p-(y'-z')$. So, $(y'-z') | \delta$ or $(p-q) | \delta$ and $\alpha + \gamma + \delta = \dot{p}$. Nevertheless, as $b_j^{(i_j)}$ varies in the set $\mathcal{V} \setminus \{0^{(2)}, p^{(4)}\}$ of cardinality $2p$, and $\alpha + \beta + \gamma + \delta$ represents the number of times that $b_j^{(i_j)}$ changes either its value or/and position, then $\alpha + \beta + \gamma + \delta \leq 2p$. Since $\beta \geq 1$, it follows that $\alpha + \gamma + \delta < 2p$; and jointly with $\alpha + \gamma + \delta = \dot{p}$, we get $\alpha + \gamma + \delta = p$. Then $\delta p = p \cdot (p-q)$ and thus $\delta = p-q$. 

Definitely, the sequence will be determined by $\alpha$ blocks of 11 iterations; $\beta = \delta$ blocks of 22 iterations; $\gamma$ blocks of 11 iterations; and $\delta$ blocks of 10 iterations. Thus, the period of the sequence is 
\begin{eqnarray*}
N &=& \alpha \cdot 11 + \delta \cdot 22 + \gamma \cdot 11 + \delta \cdot 10  = (\alpha + 2\delta + \gamma)\cdot 11 + \delta \cdot 10 \\
&=& (p + p-q)\cdot 11 + (p-q)\cdot 10 = (2p-q)\cdot 11 + (p-q)\cdot 10.
\end{eqnarray*}

\end{proof}

\textbf{Case B:} Assume that at least one of the initial conditions is not rational. Firstly, by Lemma \ref{prop_xzy1} and Lemma \ref{prop_xzy2}, we will have periodicity when $x - (y-z) = tx - s(y-z)$ holds for a certain $t, s \geq 0$. That relation yields to $\frac{x}{y-z} = \frac{s-1}{t-1} \in \mathbb{Q}.$ Therefore, $\frac{x}{y-z} \in \mathbb{Q}$ is a necessary condition to assure periodicity. Let $\frac{x}{y-z} = \frac{p}{m}$ with $\gcd(p,m) = 1$. Then, by Proposition \ref{Prop_scalar}, we can modify the initial conditions of the sequence as follows
\begin{eqnarray*}
(x, x - (y-z), z, y) &\longrightarrow& \left( \frac{x}{y-z}, \frac{x}{y-z} - 1, \frac{z}{y-z}, \frac{y}{y-z} \right) \\ &\longrightarrow& \left( \frac{p}{m}, \frac{p}{m} - 1, \frac{z}{y-z}, \frac{y}{y-z} \right) \longrightarrow (p, p-m, z', y').
\end{eqnarray*}

\noindent with $z' = m \cdot \frac{z}{y-z}$, $y' = m \cdot \frac{y}{y-z}$ and $\gcd(p,m) = 1$, which implies $\gcd(p, p-m) = 1$. Moreover, $y'-z' = m \cdot \frac{y}{y-z} -  m \cdot \frac{z}{y-z} = m \in \mathbb{N}$. Thus, if we denote by $q = p-m$, the initial conditions are $(p,q,z',y')$ with gcd$(p,q)=1$ and $y'-z'\in \mathbb{N}$; and the reasoning of Proposition \ref{N_controversial2} is valid for this case too, namely, the period of the sequence is $N=(2p-q)\cdot 11 + (p-q)\cdot 10$.

\begin{remark} \label{Remark_controv2}
Notice that the coefficients of the decomposition, that is, $a=p-q$ and $b=2p-q$ verify the conditions $b \geq 2a + 1$ and $\mathrm{gcd}(a,b) = 1$. Conversely, every number of the form $\widetilde{N} = 10\cdot a + 11 \cdot b$ with $a,b\in \mathbb{N}$, $b\geq 2a +1$ and $\gcd(a,b)=1$ can be written as $\widetilde{N} = 10\cdot (p-q) + 11 \cdot (2p-q)$. It only suffices to take $p = b-a$ and $q = b - 2a$.
\end{remark}

\begin{corollary} \label{Cor_xzy0}
Given the initial conditions $(x,z,y,0)$ with $x \geq y > z > 0$, they generate a periodic sequence under Equation (\ref{Eq:G4}) if and only if $\frac{x}{y-z} \in \mathbb{Q}$. In this case, the period of the sequence is $N = 10 \cdot a + 11 \cdot b$, for some $a,b \in \mathbb{N}$ with $\gcd(a,b)=1$ and $b \geq 2a + 1$.
\end{corollary}

\subsection{Intersection between cases $C_i'$s} \label{Sec_intersection}

Firstly, remember the cases $C_i$, defined for the study of the evolution of a tuple under Equation (\ref{Eq:G4}):

\textbf{Case 1 (}$\mathbf{C_1}$\textbf{):} $x_1 \geq x_2 \geq x_4 \geq x_3$.

\textbf{Case 2 (}$\mathbf{C_2}$\textbf{):} $x_1 \geq x_3 \geq \max\{x_2, x_4\}$ with $x_3 \geq x_2 + x_4$.

\textbf{Case 3 (}$\mathbf{C_3}$\textbf{):} $x_1 \geq x_3 \geq \max\{x_2, x_4\}$ with $x_3 \leq x_2 + x_4$. 

\textbf{Case 4 (}$\mathbf{C_4}$\textbf{):} $x_1 \geq x_4 \geq x_2 \geq x_3$.

\textbf{Case 5 (}$\mathbf{C_5}$\textbf{):} $x_1 \geq x_4 \geq x_3 \geq x_2$.

As the different inequalities are not strict, it may be possible that a tuple $(x_1, x_2, x_3, x_4)$ verify the conditions of two different cases. Notice that this happens when at least one of the terms are repeated or when $x_3 = x_2 + x_4$, which means that the initial conditions satisfy the cases $C_2$ and $C_3$ simultaneously. However, if this happens, the tuple will be of the form $(x,z,y,y-z)$, which is equivalent to the controversial case $(x,x,z,y)$. Indeed, if we compute the tuple under Equation (\ref{Eq:G4}), $ x, z, y, y-z, y-x, y-z, -z, 0, x-z,$ $x-y, x, x, z, y. $ So, $(x,z,y,y-z) \sim (x,x,z,y)$ which, in turn, is equivalent to $(x,z,y,0)$ by Proposition \ref{prop_controversial} and, hence, this case has been already analysed in Subsection \ref{Sec_xzy0}.   

Now we will study the intersections of cases $C_i'$s depending on the number of terms that are repeated. In what follows we will assume $x > y > z \geq 0$.
\begin{itemize}
    \item \textbf{The four terms are equal:}
        \subitem $\circ$ $(x,x,x,x):$ it generates an $11$-cycle (see Proposition \ref{Prop_monotonic}).
    \item \textbf{Three terms are equal:}
        \subitem $\circ$ $(x,x,x,y)$ and $(x,y,y,y)$: they are $11$-cycles due to the monotonicity of the terms.
        \subitem $\circ$ $(x,x,y,x)$: it is a controversial case. See Sections \ref{Sec_controversial} and \ref{Sec_specific}.
        \subitem $\circ$ $(x,y,x,x)$: under applying Equation (\ref{Eq:G4}), we obtain $x, y, x,$ $x,$ $0, x-y$. Then $(x,y,x,x) \sim (x,x,0,x-y)$, which behaves as the controversial case $(x,x,0,y')$. 
    \item \textbf{Two terms are equal:}
        \subitem $\circ$ $(x,x,y,z)$, $(x,y,y,z)$, $(x,y,z,z)$ and $(x,x,y,y)$: they are $11$-cycles due to the monotonicity of the terms. 
        \subitem $\circ$ $(x,y,x,z)$, $(x,x,z,y)$ and $(x,z,y,x)$: they are controversial cases, already analyzed in Section \ref{Sec_controversial}.
        \subitem $\circ$ $(x,y,z,x)$: we iterate the tuple under Equation (\ref{Eq:G4}) - $x, y, z, x,$ $0, x-y, x-z$. Then $(x,y,z,x) \sim (x,0,x-y,x-z)$, which behaves as the controversial case $(x,0,z',y')$.
        \subitem $\circ$ $(x,z,x,y)$: by iterating the initial conditions, $x, z, x, y,$ $0,$ $x-z$. So $(x,z,x,y) \sim (x,y,0,x-z)$. If $y\geq x-z$, then $(x,y,0,x-z)$ behaves as the controversial case $(x,y',0,z')$; on the contrary, if $y < x-z$, then $(x,y,0,x-z)$ behaves as the controversial case $(x,z',0,y')$.
        \subitem $\circ$ $(x,y,x,y)$: the sequence evolves as $x, y, x, y,$ $0, x-y$. Then, $(x,y,x,y) \sim (x,y,0,x-y)$. If $y\geq x-y$, then $(x,y,0,x-y)$ behaves as the controversial case $(x,y',0,z')$; on the other hand, if $y< x-y$, then it behaves as $(x,z',0,y')$, which is a controversial case too.
        \subitem $\circ$ $(x,y,y,x)$: we iterate the tuple, $x, y, y, x,$ $0, x-y, x-y$. Then $(x,y,y,x) \sim (x,0,x-y,x-y)$, which is like the controversial case $(x,0,y',y')$.
        \subitem $\circ$ $(x,y,z,y)$, $(x,z,z,y)$, $(x,z,y,y)$ and $(x,z,y,z)$. These cases have not been considered previously, so they must be analyzed.
\end{itemize}

Therefore, we only have to study the tuples $(x,y,z,y)$, $(x,z,z,y)$, $(x,z,y,y)$ and $(x,z,y,z)$. Nevertheless, we state that the four tuples are equivalent.

\begin{proposition}
Let $x,y,z$ be real numbers such that $x \geq y \geq z \geq 0$. Then we have the following relations of equivalence: $$(x,y,z,y) \sim (x,z,z,y) \sim (x,z,y,y) \sim (x,z,y,z).$$
\end{proposition}

\begin{proof}
We just have to compute the terms under Equation (\ref{Eq:G4}).
$$ \begin{matrix} \textsf{x, y, z, y,} & y-x, & 0, & y-z, & -z, & x-z, & x-z, & x-y,& \\\textsf{ x, z, z, y,} & y-x, & y-z, & y-z, & -z, & x-z, & x-y, & x-y, & \\ \textsf{x, z, y, y,} & y-x, & y-z, & 0, & -z, & x-z, & x-y, & x-z,  & \textsf{x, z, y, z}. \end{matrix} $$
\end{proof}

Furthermore, we claim that the initial conditions $(x,y,z,y)$ are a particular case of the controversial $(x,z,y,0)$. Indeed, by the simplifications made in Subsection \ref{Sec_xzy0}, $(x,z,y,0) \sim (x, x - (y-z), z, y)$. Then, if $x+z = 2y$, we get the particular case $(x,y,z,y)$. Thus, we can apply Corollary \ref{Cor_xzy0}. 

Summarizing, if we have initial conditions holding more than one case $C_i$, $i=1, \ldots, 5$, that generate a periodic sequence, we can assure that its period, $N$, is either $1,8,11$, or follows the pattern $N = 10\cdot a + 11\cdot b$, with $\gcd(a,b) = 1$ and $b \geq 2a +1$. 

\subsection{Condition U} \label{Sec_diagram}

Finally, we only have to discuss the initial conditions that generate under Equation~(\ref{Eq:G4}) a periodic sequence that follows Diagram \ref{Diagrama} unambiguously, that is, the initial tuple $(x_1,x_2,x_3,x_4)$ and each tuple after a block of ten or eleven iterations verify one and only one case $C_i$.  So, if the tuple $(x_1,x_2,x_3,x_4)$ only admits the conditions of one case $C_i$, while applying Diagram \ref{Diagrama}, it will exist a unique well-defined cycle. Therefore, in the rest of the subsection we will assume the following condition:

\begin{quote} 
    \textbf{Condition U:} A tuple of initial conditions $(x_1,x_2,x_3,x_4)$ only verifies one case $C_i$ and have exactly a unique movement according to Diagram \ref{Diagrama}.
\end{quote}

So, by Condition U we understand that it does not exist ambiguity when we go from one case $C_i$ to another $C_j$. However, if at some point of the circuit we could have the possibility of choosing more than one case $C_i$, then we would face one of the intersection cases studied before in Subsection \ref{Sec_intersection}.  

To illustrate Condition U, let us consider the initial conditions $x_1 = 8,$ $x_2 = 2$, $x_3 = 1$ and $x_4 = 5$. Observe that $(8,2,1,5)$ verifies the conditions of case $C_4$. If we compute this tuple under Equation $(\ref{Eq:G4})$, we get a $43$-cycle:
{\small{\begin{eqnarray*}
& \textbf{8}, \ \ \textbf{2}, \ \ \textbf{1}, \ \ \textbf{5}, \ \ -3, \ \ 3, \ \ 4, \ \ -1, \ \ 7, \ \ 4, \ \ 3, \\
& \mathbf{8}, \ \ \mathbf{1}, \ \ \mathbf{4}, \ \ \mathbf{5}, \ \ -3, \ \ 4, \ \ 1, \ \ -1, \ \ 7, \ \ 3, \ \ 6, \\
& \mathbf{8}, \ \ \mathbf{1}, \ \ \mathbf{5}, \ \ \mathbf{2}, \ \ -3, \ \ 4, \ \ -1, \ \ 2, \ \ 7, \ \ 3,   \\
& \mathbf{8}, \ \ \mathbf{6}, \ \ \mathbf{1}, \ \ \mathbf{5}, \ \ -2, \ \ -1, \ \ 4, \ \ -1, \ \ 6, \ \ 7, \ \ 3, \ \ \mathbf{8}, \ \ \mathbf{2}, \ \ \mathbf{1}, \ \ \mathbf{5}.  
\end{eqnarray*}}}
As it can easily be seen, after the corresponding block of 10 or 11 iterations, the new tuple satisfy one and only one case. Indeed, $(8,2,1,5)$ is in $C_4$; $(8,1,4,5)$ verifies the conditions of $C_5$; $(8,1,5,2)$ satisfies the case $C_2$; and $(8,6,1,5)$ is in $C_1$. 

Now a natural question arises, can we have a cycle that always stays in the same case? As we can appreciate in Figure \ref{Diagrama}, Cases $C_1$ and $C_3$ are the only cases where after 11 iterations we can stay in the same case. The answer is only positive for $11$-cycles or the equilibrium point.

\begin{proposition} \label{Prop_loop}
Let $(x_n)$ be a periodic sequence of period $p \geq 12$. Assume that the initial conditions satisfy $C_1$ or $C_3$. Then, there exists an $m>1$ such that after $11\cdot m$ iterations, the new tuple $(x_{1+11m},x_{2+11m},x_{3+11m},x_{4+11m})$ do not satisfy the same inequalities as the initial terms in $C_1$ or $C_3$.
\end{proposition}

\begin{proof}
We will prove the case where the initial conditions satisfy $C_1$. Case $C_3$ is analogous and it will be left for the reader.

Let us assume the opposite, i.e., the tuple $(x_{1+11j},x_{2+11j},x_{3+11j},x_{4+11j})$ satisfy $C_1$ for every $j \in \mathbb{N}\cup \{0\};$ in particular, $x_1\geq x_2\geq x_4\geq x_3$.  By induction, it is easy to check that after $m$ blocks of 11 iterations we obtain the tuple $(x_1,x_2+m\cdot x_3-m\cdot x_4,x_3,x_4)$. If we force periodicity, the equality $x_2 + m\cdot x_3 - m\cdot x_4 = x_2$ must hold, but that relation yields to $x_3 = x_4$, what means that the initial conditions are monotonic. Therefore we have an 11-cycle if $x_1>0$ or the equilibrium point if $x_1=0$.
\end{proof}

From now on, we will consider Figure~\ref{Diagrama} as an oriented graph, $G=(V,U)$, where $V=\{C_1,C_2,C_3,C_4,C_5\}$ is a finite set and $U \subset V \times V$. The elements of $V$ are the vertices of the oriented graph $G$ and each element $(C_i, C_j) \in U$ will be called an arrow from $C_i$ to $C_j$. Thus, the elements of $U$ are the arrows of $G$. A path that always visits the same vertex is called a loop (our graph $G$ only admits two loops, one in $C_1$, and another in $C_3$). A route is a circuit that visits each vertex once, except the possibility of having a loop. We denote them by $R_i$. Also, in view of Diagram \ref{Diagrama}, we assume without loss of generality that the initial conditions satisfy case $C_4$, and then periodicity will hold when the corresponding tuple verifies again the conditions of the same case. Our graph only admits the following routes (see Diagram~\ref{Diagrama}):

$R_1: C_4 \longrightarrow C_5 \longrightarrow C_2 \longrightarrow C_1 \rightarrow \ldots \rightarrow C_1 \longrightarrow C_4.$ 

$R_2: C_4 \longrightarrow C_5 \longrightarrow C_2 \longrightarrow C_4.$ 

$R_3: C_4 \longrightarrow C_5 \longrightarrow C_3 \rightarrow \ldots \rightarrow C_3 \longrightarrow C_2 \longrightarrow C_1 \rightarrow \ldots \rightarrow C_1 \longrightarrow C_4.$

$R_4: C_4 \longrightarrow C_5 \longrightarrow C_3 \rightarrow \ldots \rightarrow C_3 \longrightarrow C_2 \longrightarrow C_4. $

The length of each route will be the number of iterations needed to go from the initial conditions until the final step. In our case, $$|R_1| = 43 + 11\cdot m,\ \ |R_2| = 32, \ \ |R_3| = 54 + 11\cdot (m+n), \ \  |R_4| = 43 + 11\cdot n,$$ where $m\geq 0$ and $n\geq 0$ are the times that the vertices $C_1$ and $C_3$ are repeated in the loops. A cycle in $G$ is a trajectory that starts and ends in the same vertex. Notice that a cycle is a finite concatenation of the routes $R_i$. We assume, without loss of generality, that \textit{we start and end in $C_4$.}

In this sense, the evolution of a tuple $(x_1,x_2,x_3,x_4)$, with $x_1 \geq x_4 \geq x_2 \geq x_3$, under each route, will be: 
\begin{eqnarray*}
R_1: (x_1,x_2,x_3,x_4) &\Longrightarrow& (x_1, x_1 + x_2 - (m+2)\cdot (x_4 - x_3), x_3, x_4); \\
R_2: (x_1,x_2,x_3,x_4) &\Longrightarrow& (x_1, x_1 + x_2 - (x_4 - x_3), x_3, x_4); \\
R_3: (x_1,x_2,x_3,x_4) &\Longrightarrow& (x_1, x_1 + x_2 - (m+n+3)\cdot (x_4 - x_3), x_3, x_4); \\
R_4: (x_1,x_2,x_3,x_4) &\Longrightarrow& (x_1, x_1 + x_2 - (n+2)\cdot (x_4 - x_3), x_3, x_4).
\end{eqnarray*}

We will see in detail the evolution of $R_1$ so that the reader can perfectly understand the behaviour of the routes. The other cases are analogous and will be omitted. For instance, let $(x_1, x_2, x_3, x_4)$ be a tuple of positive terms verifying the conditions of $C_4$, that is, $x_1 \geq x_4 \geq x_2 \geq x_3$. We consider Table \ref{Table_cases} to develop what follows:
\begin{eqnarray*}
C_4: (x_1, x_2, x_3, x_4) &\xrightarrow{11}& (x_1, x_3, x_4 + x_3 - x_2, x_4) \in C_5 \\
&\xrightarrow{11}& (x_1, x_3, x_4, x_2) \in C_2 \\
&\xrightarrow{10}& (x_1, x_1 + x_2 + x_3 - x_4 , x_3, x_4) \in C_1 \\
&\xrightarrow{11}& (x_1, x_1 + x_2 + 2(x_3 - x_4), x_3, x_4) \in C_1 \\
&\xrightarrow{11}& \ldots \\
&\xrightarrow{11}& (x_1, x_1 + x_2 + (m+1)(x_3-x_4), x_3, x_4) \in C_1 \\
&\xrightarrow{11}& (x_1, x_1 + x_2 + (m+2)(x_3-x_4), x_3, x_4) \in C_4,
\end{eqnarray*}
\noindent where $m$ denotes the loops that occur in $C_1$. Notice that $x_1 + x_2 + jx_3 \geq (j+1)x_4$ for $j=1,\ldots,m$ and $x_1 + x_2 + (m+1)x_3 \leq (m+2)x_4$.

\subsubsection{Evolution by concatenation of routes}

Now, we will analyse the behaviour of a periodic sequence $(x_n)$, whose initial conditions follow a cycle which is a concatenation of the routes $R_i$. Thus, if we denote by $A_i$ the number of times that each route $R_i$ appears, $i=1,2,3,4$, we can  precise the evolution of the initial tuple $(x_1,x_2,x_3,x_4)$ along the cycle. Notice that when we have described the routes $R_i$, the second term is the only one that changes at the end of each route. So, to simplify the notation, we will only write the evolution of the second term. The study of this fact, which is based on induction, is mechanical and will be omitted.

Firstly, we start analysing the evolution after $A_1$ routes of $R_1$. Here, the term $x_2$ becomes into $A_1\cdot x_1 + x_2 - (m_1+ \ldots + m_{A_1} + 2\cdot A_1)\cdot(x_4 -  x_3)$, being $m_i \ (i= 1, \dots, A_1)$ the number of loops in $C_1$ that occurs in each route.

Secondly, the evolution after $A_2$ routes of $R_2$ transforms $x_2$ into $A_2 \cdot x_1 + x_2 - A_2\cdot(x_4 - x_3).$ 

Next, we focus on the case of $A_3$ routes of the type $R_3$. Let us denote $\tilde{m}_i \ (i= 1, \dots, A_3)$ and $\tilde{n}_j \ (j= 1, \dots, A_3)$  as the numbers of loops in $C_1$ and $C_3$, respectively. Then, $x_2$ will evolve to $A_3\cdot x_1 + x_2 - (\tilde{m}_1+\ldots + \tilde{m}_{A_3} + \tilde{n}_1 + \ldots + \tilde{n}_{A_3} + 3\cdot A_3)\cdot(x_4 - x_3)$. 

Finally, we study the case of $A_4$ routes of the type $R_4$. In what follows, $n_i \ (i= 1, \dots, A_4)$ denotes the number of loops in $C_3$ that occurs in each route. Then, $x_2$ becomes into $A_4 \cdot x_1 + x_2 - (n_1+ \ldots + n_{A_4} + 2\cdot A_4)\cdot (x_4 - x_3).$

Now, after studying each case in particular, we are able to express the general case. Observe that the order of the routes does not affect in the determination of the period. For instance, the final effect of $R_1, R_1, R_2, R_2$ is the same as $R_1, R_2, R_1, R_2$.

\begin{corollary} \label{Cor_routes}
If the periodic sequence $(x_n)$ follows $A_1$ routes $R_1$, $A_2$ routes $R_2$, $A_3$ routes $R_3$ and $A_4$ routes $R_4$, the second term will end as: $$(A_1 + A_2 + A_3 + A_4)\cdot x_1 + x_2 - (H + 2A_1 + 2A_4 + A_2 + 3\cdot A_3)\cdot (x_4 - x_3),$$\noindent where $H = m_1+ \ldots + m_{A_1} + n_1+ \ldots + n_{A_4} + \tilde{m}_1+\ldots + \tilde{m}_{A_3} + \tilde{n}_1 + \ldots + \tilde{n}_{A_3}$.
\end{corollary}

Therefore, if a periodic sequence follows the pattern of routes of Corollary~\ref{Cor_routes}, then the second term of the tuple must equal $x_2$ and we get the following condition
\begin{equation} \label{x1_rutas}
x_1 = \frac{H + 2A_1 + 2A_4 + A_2 + 3A_3}{(A_1 + A_2 + A_3 + A_4)}\cdot (x_4 - x_3).
\end{equation}

So if we consider the length of the routes, since the initial conditions verify Condition U, then the period of the sequence will be determined by $N$, where $N$ is defined by the minimal concatenation of routes needed to achieve periodicity ($A_j$ routes $R_j$, $j=1,2,3,4$):
\begin{eqnarray*}
 N &=& A_1\cdot |R_1| + A_2\cdot |R_2| + A_3\cdot |R_3| + A_4\cdot |R_4| \\
 &=& (A_1 + A_2 + A_3 + A_4) \cdot 10 + (3A_1 + 2 A_2 + 4A_3 + 3 A_4 + H)\cdot 11.
\end{eqnarray*}

We will call $A = A_1 + A_2 + A_3 + A_4$ and $B = 3A_1 + 2 A_2 + 4A_3 + 3 A_4 + H$. In this sense, $N = 10\cdot A + 11 \cdot B$ with $A,B \in \mathbb{N}$.

\subsubsection{The periods of sequences satisfying Condition U}

Now we will take into account the behaviour of the different routes $R_i$ to analyse a particular case of the initial conditions, $(x_1,x_2,x_3,x_4)$, verifying the inequalities of $C_4$. This study will allow us to determine a necessary condition for the periods of Equation (\ref{Eq:G4}).

Beforehand, we will homogenize the notation used to describe the routes $R_i$. Let us denote $$\delta_1 = m+2, \ \ \delta_2 = 1, \ \ \delta_3 = m+n+3, \ \ \delta_4 = n+2,$$ \noindent where $m$ and $n$ represented the number of loops that take place in $C_1$ and $C_3$ respectively. In this sense, the second term of the initial conditions will change by a route $R_i$ as 
\begin{equation} \label{tildex2}
    \tilde{x}_2 = x_1 + x_2 - \delta_i(x_4-x_3). 
\end{equation}
Moreover, it is easy to check that the length of the route $R_i$ will be $32 + 11\cdot (\delta_i-1)$, $i=1,2,3,4.$

\begin{proposition} \label{prop_gcd}
Let $p, q$ be natural numbers with $q \geq 2p + 1$ and $\mathrm{gcd}(p,q)=1$. Let $(x_1, x_2, x_3, x_4)$ be a tuple of real numbers verifying Condition U such that $x_1 = \frac{q-p}{p}(x_4-x_3)$ and $x_1 \geq x_4 \geq x_2 \geq x_3$ (case $C_4$). Then these initial conditions determine under Equation (\ref{Eq:G4}) a periodic sequence described by $p$ routes $R_i$. Furthermore, its period is $N = 10\cdot p + 11\cdot q$.
\end{proposition}

\begin{proof}

Firstly, observe that after $A$ routes, with $A \in \mathbb{N}$, the second term $x_2$ will become $Ax_1+x_2 - \sum_{j=1}^A\delta^j(x_4 - x_3)$, where $\delta^j$ denotes the $\delta_i$ associated to the $j$-th route of the cycle, with $j = 1, \ldots, A$.  Then periodicity will take place if $Ax_1 = \sum_{j=1}^A \delta^j(x_4-x_3)$. We use $x_1 = \frac{q-p}{p}(x_4-x_3)$ by hypothesis and that $x_4 > x_3$ (in the contrary, $x_4 = x_2 = x_3$ and we will have an $11$-cycle). In this sense, we obtain $A \cdot \frac{q-p}{p} = \sum_{j=1}^A \delta^j.$ However, $\sum_{j=1}^A \delta^j$ is a natural number, while $\frac{q-p}{p}$ is not, since $\mathrm{gcd}(p,q) = 1$. Thus, if the initial conditions generate a periodic sequence, then $A =\dot{p}$. Indeed, we will see that $A = p$. But, previously, we need to guarantee that, in fact, the sequence is periodic.

Next, we focus again on the evolution of the initial conditions under the routes $R_i$. Notice that, independently of the route, the conditions of Case $C_4$ force the second term to be less or equal to the fourth, that is, $$x_1 + x_2 - \delta_i(x_4-x_3) \leq x_4.$$

Equivalently, we substitute $x_1 = \frac{q-p}{p}(x_4-x_3)$, add and subtract the third term $x_3$: $\frac{q-p}{p}(x_4-x_3) + x_2 - (\delta_i+1)(x_4-x_3) - x_3 \leq 0.$ Now we divide by $x_4 - x_3$ (remember $x_4 > x_3$) and denote $\mu = \frac{x_2-x_3}{x_4-x_3} \in [0,1]$. Hence, for every route $R_i$, we have 
\begin{equation} \label{Desigualdad1}
    \frac{q-p}{p} + \mu - 1 \leq \delta_i.
\end{equation}

Now we will analyse each route in order to find an upper bound for the integer $\delta_i$, $i=1,\ldots,4$. Firstly, we will consider the routes $R_1$ and $R_3$. In both cases, if we do 11 iterations backwards from the final tuple $(x_1, x_1 + x_2 - \delta_i(x_4-x_3), x_3, x_4)$, we will get $(x_1, x_1 + x_2 - (\delta_i-1)(x_4-x_3), x_3, x_4)$ that verifies the case $C_1$ (take into account that the last step in routes $R_1$ and $R_3$ is $C_1 \rightarrow C_4$, and that $(x_n)$ verifies Condition U). Thus, the inequalities that define $C_1$ must hold. In particular, the second term must be greater or equal to the fourth, which implies $$x_1 + x_2 - (\delta_i-1)(x_4-x_3) \geq x_4.$$

We proceed analogously to Inequality (\ref{Desigualdad1}) and obtain
$$\frac{q-p}{p}(x_4-x_3) + x_2 - (\delta_i-1)(x_4-x_3)-x_4 \geq 0,$$
$$\frac{q-p}{p}(x_4-x_3) - \delta_i(x_4-x_3) + x_2 - x_3 \geq 0. $$

We divide by $x_4-x_3$ and for the routes $R_1$ and $R_3$ we have 
\begin{equation} \label{Desigualdad2}
    \delta_i \leq \frac{q-p}{p} + \mu.
\end{equation}

Also, notice that Inequality (\ref{Desigualdad2}) is valid for $R_2$, because in that case $\delta_i = \delta_2 = 1$ and $\frac{q-p}{p} \geq 1$, since $q \geq 2p+1$. Finally, we consider the route $R_4$. If we do 10 iterations backwards (now, the last step in $R_4$ is $C_2 \rightarrow C_4$) from the final tuple $(x_1, x_1 + x_2 - \delta_i(x_4-x_3), x_3, x_4)$, we will obtain the terms $(x_1, x_3, x_4, x_2 - (\delta_i-1)(x_4-x_3))$ which verify the conditions of $C_2$. We do again 11 iterations backwards to get the tuple $(x_1, x_3, x_4, x_2 - (\delta_i-2)(x_4-x_3))$ that verify $C_3$. In this sense, the inequalities that define $C_3$ force the sum of the second and fourth term to be greater than or equal to the third. So, $x_2 - (\delta_i -2)(x_4 - x_3) + x_3 \geq x_4,$ from where we can deduce $$x_2 - (\delta_i -1)(x_4 - x_3) \geq 0 \ \ \ \text{and} \ \ \ x_2 + x_4 - x_3 \geq \delta_i(x_4-x_3).$$

However, as $x_1 \geq x_2$, we get $x_1 + x_4 - x_3 \geq \delta_i(x_4-x_3)$ and, therefore, $$\frac{q-p}{p}(x_4-x_3) + x_4 - x_3 \geq \delta_i (x_4 - x_3).$$

So, we achieve
\begin{equation} \label{Desigualdad3}
    \frac{q-p}{p} + 1 \geq \delta_i.
\end{equation}

To sum up, for every route $R_i$, we can synthesise Inequalities (\ref{Desigualdad1}), (\ref{Desigualdad2}) and (\ref{Desigualdad3}) as
$$\frac{q-p}{p} + \mu - 1 \leq \delta_i \leq \frac{q-p}{p} + 1.$$

Since $\delta_i$ is a natural number, it only admits two possible values, that are $\delta_i = \left[\frac{q-p}{p}\right] \ \ \ \text{or} \ \ \ \delta_i = \left[\frac{q-p}{p}\right] + 1,$ where $[\cdot]$ represents the integer part of a number. Thus, after a route $R_i$, taking into account (\ref{tildex2}) and the condition $x_1 = \frac{q-p}{p} (x_4 - x_3)$, the second term $x_2$ will become either in 
\begin{equation} \label{new_x2}
\tilde{x}_2 = x_2 + \left( \frac{q-p}{p} - \left[\frac{q-p}{p} \right] \right)(x_4-x_3),
\end{equation}

or either in
\begin{equation} \label{new_x2_bis}
\tilde{x}_2 = x_2 + \left( \frac{q-p}{p} - \left[\frac{q-p}{p} \right] - 1 \right)(x_4-x_3).
\end{equation}

It should be emphasized that $\frac{q-p}{p} - \left[\frac{q-p}{p} \right] = \left\{ \frac{q-p}{p} \right\}$, where $\{\cdot\}$ denotes the fractional part of a number. Hence $\left\{ \frac{q-p}{p} \right\} = \frac{r}{p}$, where $r$ is the rest of the division of $q-p$ by $p$. So, we can reduce expressions (\ref{new_x2}) and (\ref{new_x2_bis}) into \begin{equation} \label{newlabel}
    \tilde{x}_2 = x_2 + \frac{r}{p}(x_4-x_3) \ \ \ \text{or} \ \ \ \tilde{x}_2 = x_2 - \left(1 - \frac{r}{p} \right)(x_4 - x_3).
\end{equation}

But, after a route, which is the factor that determines if the second term is defined by one expression or the other? The key is that the new tuple $(x_1, \tilde{x}_2, x_3, x_4)$ must verify the conditions of $C_4$. In particular, $x_4 \geq \tilde{x}_2 \geq x_3$. Let us see that these inequalities are always verified by one and only one of the two possible options, i.e., is not possible to fulfill at the same time the inequalities $x_4 \geq \tilde{x}_2 \geq x_3$ by the two values $\tilde{x}_2$ appearing in (\ref{newlabel}).
\begin{itemize}
    \item If $x_4\geq x_2 + \frac{r}{p}(x_4-x_3)$, then $x_3 \geq x_2 - \left(1-\frac{r}{p}\right)(x_4-x_3)$. So the second option in (\ref{newlabel}) does not verify the conditions characterising $C_4$. Notice that $x_2 + \frac{r}{p}(x_4-x_3) \geq x_3$ trivially. 
    \item If $x_2 + \frac{r}{p}(x_4-x_3) > x_4$ and, thus, the first expression in (\ref{newlabel}) does not hold conditions in $C_4$, then we have $x_2 + \frac{r}{p}(x_4-x_3) > x_4 - x_3 + x_3;$ and $x_2 - \left(1-\frac{r}{p}\right)(x_4 - x_3) > x_3$. So the second option satisfies the conditions of $C_4$, since $x_4 \geq x_2 - \left(1-\frac{r}{p}\right)(x_4 - x_3)$ trivially. 
\end{itemize}

Definitely, after one route $R_i$, we find 
\begin{equation} \label{Neq_x2}
    \tilde{x}_2 = \left\{ \begin{matrix} x_2 + \frac{r}{p}\cdot d & \text{if} \ x_2 + \frac{r}{p}\cdot d \leq x_4, \\ x_2 + \frac{r}{p}\cdot d - d &  \text{otherwise,} \end{matrix} \right. 
\end{equation}

\noindent where $d = x_4 - x_3$. So, after $p$ routes, the second term $\tilde{x}_2$ will be $x_2 + p\cdot \frac{r}{p} \cdot d - \alpha \cdot d,$ with $\alpha \in \mathbb{Z}$ and $0\leq \alpha \leq p.$ Notice that $\alpha$ represents the number of times that $\tilde{x}_2$ takes the second value of (\ref{Neq_x2}) after one route. Furthermore, it must be hold that
$$x_3 \leq x_2 + r\cdot d - \alpha \cdot d \leq x_4,$$
$$\alpha \cdot d \leq x_2 - x_3 + r \cdot d \leq (\alpha + 1) \cdot d.$$

Dividing by $d$,
$$\alpha \leq \frac{x_2-x_3}{d} + r \leq \alpha + 1, $$

\noindent with $0\leq \frac{x_2-x_3}{d} \leq 1$. Therefore, as $\alpha$ and $r$ are integers, we deduce that $\alpha = r$, except when $x_2 = x_3$, where apart from $\alpha = r$, it can also be hold $\alpha = r-1$. However, if $\alpha = r-1$, then $x_2=x_3$ and the initial conditions would verify simultaneously cases $C_4$ and $C_5$ in contradiction with Condition U. Thus, $\alpha = r$.

If $\alpha = r$, after $p$ routes, the second term will verify $x_2 + p \cdot \frac{r}{p} \cdot d - rd = x_2.$ Thus, after $p$ routes, the initial terms $(x_1,x_2,x_3,x_4)$ with $x_1 \geq x_4 \geq x_2 \geq x_3$, $x_1 = \frac{q-p}{p}(x_4-x_3)$ where $q \geq 2p+1$ and $\mathrm{gcd}(p,q) = 1$ generate a periodic sequence. Moreover, as we have seen that the number of routes needed to achieve periodicity must be $A = \dot{p}$, we conclude $A = p$.
   
Finally, we will see the period of the sequence. For that, remember that the length of each route $R_i$ is $32 + 11 \cdot (\delta_i-1)$. Therefore, as we achieve periodicity after $p$ routes $R_i$, we obtain that the period is 
$$N = \sum_{j=1}^p(32+11\cdot(\delta^j-1)) = 21p + 11 \sum_{j=1}^p \delta^j,$$

\noindent where $\delta^j$ denoted the $\delta_i$ associated to the $j$-th route of the cycle. This sum is easy to compute, since after $A = p$ routes, the second term will be $p\cdot x_1 + x_2 - \sum_{j=1}^p\delta^j(x_4-x_3)$ and if we equal to $x_2$ (we have periodicity), $p\cdot x_1 = \sum_{j=1}^p\delta^j(x_4-x_3).$ So, $$p\cdot \frac{q-p}{p}(x_4-x_3) = \sum_{j=1}^p\delta^j(x_4-x_3) \ \ \ \text{and} \ \ \ \sum_{j=1}^p\delta^j = q - p.$$

Now we substitute in the previous expression of the period and we obtain \vspace{1mm} \newline $N = 21\cdot p + 11 \sum_{j=1}^p \delta^j = 21\cdot p + 11\cdot (q-p) = 10\cdot p + 11\cdot q.$
\end{proof}

This description allows us to give a necessary condition for a certain number to be in the set of periods when Condition U is satisfied and the period is attained by a concatenation of routes. 

\begin{proposition} \label{nec}
Let $N$ be a natural number. If $N$ is the prime period of a periodic sequence described by routes $R_i$ and whose initial conditions verify Condition U, then $N$ admits a decomposition $N = 10\cdot a + 11\cdot b$ with $a, b \in \mathbb{N}$ such that $\gcd(a,b) = 1$ and $b \geq 2a + 1$.
\end{proposition}

\begin{proof}

Let $(x_n)$ be a periodic sequence of Equation (\ref{Eq:G4}) described by the routes $R_i$ where $x_1 = \max\{x_n: n\geq 1\}$. Let $N$ be its prime period. By the previous analysis (see Corollary \ref{Cor_routes} and (\ref{x1_rutas})), $N$ will admit a decomposition of the form $N = 10\cdot a + 11\cdot b$, with $a,b \in \mathbb{N}$. Moreover, the coefficients $a$ and $b$ are equal to $A_1 + A_2 + A_3 + A_4$ and $3A_1 + 2 A_2 + 4A_3 + 3 A_4 + H$, respectively, where $A_i \ (i=1,\ldots,4)$ and $H$ are the ones described in (\ref{x1_rutas}). To finish we need to prove that $\gcd(a,b) = 1$ and $b \geq 2a + 1$.

Notice that $b = 2a + A_1 + A_4 + 2A_3 + H$. We claim that $A_1 + A_4 + 2A_3 + H \geq 1$. As $A_i \geq 0$ for $i=1,\ldots,4$, and $H \geq 0$, if $A_1 + A_4 + 2A_3 + H = 0$, it means that the cycle is formed by $A_2$ routes of $R_2$, but that is not possible. Indeed, due to the description that has been made about the routes, it is easy to check by induction that the initial conditions will become $(x_1, A_2 \cdot x_1 + x_2 + A_2 \cdot (x_3-x_4), x_3, x_4).$ Then, to have periodicity, the second term must be $x_2$, which yields to $x_1 + x_3 = x_4$. As $x_1$ is the maximum and the initial tuple is nonnegative, then $x_3 = 0$ and $x_1=x_4$. Notice that it implies the equilibrium or an 8-cycle. Thus, $A_1 + A_4 + 2A_3 + H \geq 1$ which yields to $b \geq 2a + 1$.

On the other hand, given such an $N$, suppose that $\gcd(a,b)=d > 1$. Then there exist $p, q\in \mathbb{N}$ such that $a = d \cdot p$ and $b = d \cdot q$. As the sequence $(x_n)$ follows the routes $R_i$, by Equation (\ref{x1_rutas}) we know that $x_1 = \frac{b-a}{a}(x_4 - x_3)$, which means that the initial conditions that generate the sequence are $\left ( \frac{b-a}{a} (x_4 - x_3), x_2, x_3, x_4 \right )$. However, the fraction $\frac{b-a}{a}$ reduces to $\frac{q-p}{p}$, since $\frac{b-a}{a} = \frac{d\cdot q - d \cdot p}{d\cdot p} = \frac{q-p}{p}$. So, initial conditions becomes into $\left ( \frac{q-p}{p}(x_4 - x_3), x_2, x_3, x_4 \right ),$ with $x_4 \geq x_2 \geq x_3 \geq 0$. Now, $\mathrm{gcd}(p,q) = 1$ and $q\geq 2p+1$ (since $b\geq 2a+1$), so by Proposition \ref{prop_gcd}, the initial conditions $(\frac{q-p}{p}(x_4 - x_3), x_2, x_3, x_4)$ generate a periodic sequence whose period is $N_1 = 10\cdot p + 11\cdot q$. Thus, $N_1 < N$, since $d\neq 1$ and $d\cdot N_1 = 10\cdot p \cdot d + 11 \cdot q \cdot d = 10 \cdot a + 11 \cdot b = N$. Therefore, we have that $(x_n)$ is periodic with period $N_1 < N$, but we had supposed that $N$ was the prime period, so we have derived a contradiction. Then, $\mathrm{gcd}(a,b) = 1$ is a necessary condition for $N \in \mathrm{Per}(F_4)$.
\end{proof}

\subsection{Proof of the Main Theorem} \label{Sec_main}

Notice that once we have described the routes which determine Diagram \ref{Diagrama}, the controversial cases analysed in Subsection \ref{Sec_specific} and the intersection between cases $C_i$, we are able to describe precisely the set of period $\mathrm{Per}(F_4)$.

\begin{theorem} \label{Th_PerF4}
    Consider Equation (\ref{Eq:G4}) and let $\mathrm{Per}(F_4)$ be its set of periods. Then $$\mathrm{Per}(F_4) = \left\{1,8,11\right\}\bigcup \left\{10\cdot a + 11\cdot b \ | \gcd(a,b)=1, a\geq 1, b\geq 2a+1 \right\}.$$
\end{theorem}

\begin{proof}

Firstly, by Proposition \ref{P:1,8,11}, we get Per$(F_4) \cap [1,11] = \{1,8,11\}$. 
Secondly, the detailed description of Per$(F_4)$ for periods $N>11$, has been developed by the deep analysis of the controversial cases, the intersection between cases and the initial conditions satisfying Condition U. See Subsections \ref{Sec_controversial}, \ref{Sec_intersection} and \ref{Sec_diagram}, respectively.

For the controversial cases, the study of the periodic sequences generated by the initial conditions $(x,y,0,z)$ and $(x,z,y,0)$ with $x\geq y \geq z > 0$ is established by Corollaries \ref{Cor_xy0z} and \ref{Cor_xzy0}, respectively. On the other hand, the intersection between cases is developed in Subsection \ref{Sec_intersection} and we have found that it can be reduced to the casuistic of controversial cases. Finally, Proposition \ref{nec} deals with the periods of the periodic sequences generated by initial conditions holding Condition U.

Next, by Proposition \ref{prop_gcd} or Remarks \ref{Remark_controv1} and \ref{Remark_controv2}, jointly with Proposition \ref{P:1,8,11}, we conclude that Per$(F_4)$ is exactly \vspace{3mm} \newline $\left\{1,8,11\right\}\bigcup \left\{10\cdot a + 11\cdot b \ | \gcd(a,b)=1, a\geq 1, b\geq 2a+1 \right\}.$
\end{proof}

Observe that apart from describing exactly the set $\mathrm{Per}(F_4)$, we have developed an easy way to construct cycles of any period in $\mathrm{Per}(F_4)$, since once we have the decomposition of the period $N = 10\cdot a + 11\cdot b$ with $\gcd(a,b)=1$ and $b\geq 2a+1$, Proposition \ref{prop_gcd} or Propositions \ref{N_controversial1} and 
\ref{N_controversial2} give us information in order to settle the suitable initial conditions. 

\section{Evaluating the maximum of $\mathbb N\setminus\mathrm{Per}(F_4)$} \label{S:Maximum}
Once we have obtained the characterization of the set of periods as suitable combinations of multiples of $10$ and $11$,  a question that arises naturally is whether there exists or not a number $M$ such that every number greater than $M$ belongs to $\mathrm{Per}(F_4)$. We answer in affirmative and find such a number. To do it, previously we need some results on the prime numbers appearing in $\mathrm{Per}(F_4)$ as well as their multiples. Roughly speaking, we prove that almost all primes belong to $\mathrm{Per}(F_4)$ as well as their associate multiples. 

Beforehand, it is interesting to mention that if we had not had the restriction $b \geq 2a+1$, the answer would follow directly by the Diophantic Frobenius Problem. In concrete, the problem -also named coin problem in the literature- consists in finding the greatest number that can not be expressed as a linear combination with positive coefficients of a set $(a_1,\ldots,a_n)$ of natural numbers with $\gcd(a_1,\ldots,a_n) = 1$. In the particular case of two natural numbers, that greatest number is given by $a_1a_2 - a_1 - a_2$ (for more information, the reader is referred to \cite{Alfonsin}). In our case, $a_1=10$ and $a_2 = 11$, the greatest number that cannot be expressed as a linear combination of $10$ and $11$ is $89$. But as we have highlighted, the conditions $b \geq 2a+1$ and $\gcd(a,b)=1$ complicate the problem.

Let us denote the set of prime numbers by $\mathcal{P}$. In the following list, we present the first elements $p$ of $\mathcal{P}$, until $401$, and we encircle those admitting a decomposition $p=10\cdot a+11\cdot b$ with $\gcd(a,b)=1$ and $b\geq 2a+1$:

{\small{\noindent $2,3,5,7$, \circled{$11$}, $13,17,19,23,29,31,37,41,$ \circled{$43$}, $47, 53, 59, 61, 67, 71, 73, 79, 83,$ $89,$ \circled{$97$}, 
$101, 103,$ \circled{$107$}, \circled{$109$}, $113, 127,$ \circled{$131$}, $137,$ \circled{$139$}, $149,$ \circled{$151$}, $157,$ \circled{$163$}, $167$, \circled{$173$}, $179, 181, 191$, \circled{$193$}, \circled{$197$}, 
$199, 211, 223,$ \circled{$227$}, \circled{$229$}, $233,$ \circled{$239$}, \circled{$241$}, \circled{$251$}, \circled{$257$}, \circled{$263$}, 
\circled{$269$}, \circled{$271$}, $277,$ \circled{$281$}, \circled{$283$}, \circled{$293$}, \circled{$307$}, \circled{$311$}, \circled{$313$}, \circled{$317$}, \circled{$331$}, \circled{$337$}, \circled{$347$}, 
\circled{$349$}, \circled{$353$}, \circled{$359$}, \circled{$367$}, \circled{$373$}, \circled{$379$}, \circled{$383$}, 
\circled{$389$}, \circled{$397$}, \circled{$401$}.}}

For instance, $397\in\mathrm{Per}(F_4)$ since $397=297+100=11\cdot 27 + 10\cdot 10$, whereas $277\notin\mathrm{Per}(F_4)$ as the decompositions $277=187+90=11\cdot 17 + 10\cdot 9$ and $277=77+200=11\cdot 7 + 10\cdot 20$ are not allowed, because $b< 2a+1$ in both cases. We observe that we have encircled all the primes greater than $277$. In fact, this observation can be confirmed by Proposition \ref{P:losprimos}. Due to length reasons, the proof of the following results will be omitted. The interested reader can found them in \cite{LiNi}. 
\begin{proposition}\label{P:losprimos}
If $p\in\mathcal{P}$, with $p\geq 281$,  then $p\in\mathrm{Per}(F_4).$
\end{proposition}

We collect some information about the formation of elements of $\mathrm{Per}(F_4)$ once we know that a prime number $p$ belongs to the set of periods. 
\begin{proposition}\label{P:formacionP}
Let $p\in\mathcal{P}\cap \mathrm{Per}(F_4)$, with $p\geq 43.$ Assume that $p=10a+11b$, with $\gcd(a,b)=1$, $b\geq 2a+1$, $a\geq 1$.  Then:
\begin{itemize}
\item[(a)] $pq\in\mathrm{Per}(F_4)$ for all $q\geq 1$ with $\gcd(p,q)=1$ and $aq\geq 12, (b-2a)q\geq 33$. In particular,  $pq\in\mathrm{Per}(F_4)$ for all $q\geq 33$ with $\gcd(p,q)=1$. 
\item[(b)] $p^kq\in\mathrm{Per}(F_4)$ for all $k\geq 2$ and for all $q\geq 1$ with $\gcd(p,q)=1$. 
\end{itemize}
\end{proposition}

In the above results we have excluded $p=11 \in \mathrm{Per}(F_4)$. Next, we obtain the periods of $\mathrm{Per}(F_4)$ of the form $11^kq$, $k\geq 1$, and $\gcd(11,q)=1$. 
\begin{proposition}\label{P:para11}
Consider $11$, $a=0, b=1$. It holds:
\begin{itemize}
\item[(a)] $11^k q\in\mathrm{Per}(F_4)$ for all $k\geq 3$ and for all $q\geq 1$ with $\gcd(11,q)=1$.
\item[(b)] $11^2q\in\mathrm{Per}(F_4)$ for all $q\geq 3$ with $\gcd(11,q)=1$. Moreover, $11^2$ and $11^2\cdot 2$ do not belong to $\mathrm{Per}(F_4).$
\item[(c)] $11q\in\mathrm{Per}(F_4)$ whenever $\gcd(11,q)=1$ and 
$$q\in\{1\}\cup\left(\{q: q\geq 33\} \setminus\{43, 54, 76, 120\}\right).$$
\end{itemize}
\end{proposition}

\begin{corollary}\label{C:cota11} It holds that $1320 = \max\{11n: n\geq 1, 11n\notin\mathrm{Per}(F_4)\} $.
\end{corollary}
Once we have described the main properties concerning prime periods as well as a detailed study on which multiples of $11$ belong to $\mathrm{Per}(F_4),$ 
we are now interested in proving that the set $\mathcal{NP}:=\mathbb N\setminus \mathrm{Per}(F_4)$ is bounded. In fact our objective is to calculate the maximum of 
$\mathcal{NP}$.

The strategy consists in dividing the set of natural numbers, not multiple of $11$, in ten different classes, $\mathcal{C}_m:=\left\{10m + 11k, \, k\geq 0\right\}$, where we fix the value $m\in\{1,2,\ldots,10\}$. For each class $\mathcal{C}_m$, we will show that $\mathcal{NP}\cap\mathcal{C}_m$ is bounded, and from the inspection of each subset $\mathcal{NP}\cap\mathcal{C}_m$ we will deduce the maximum of $\mathcal{NP}.$ To develop the study of those classes, we use the following basic fact:
\begin{quote}\label{q:Fact}
Given a natural number $N$, $N\neq \dot{11}$, there exists a unique $m\in\{1,2,\ldots,10\}$ such that $N-10\cdot m=\dot{11}.$
\end{quote} 

As the previous part, the full development of every case $\mathcal{C}_m$ will be omitted. Observe that case $\mathcal{C}_1$ is direct since $n\in\mathcal{C}_1$ can be written as $n=10+11b$, where $a=1$ and $\gcd(b,1)=1$. To assure that $n$ is a period, it is necessary to require $b\geq 3$. Therefore, $n\in\mathrm{Per}(F_4)$ for all $n\geq 43$. Additionally, $n=10, 21, 32$ belong to $\mathcal{NP}$ and $N_1:=\max\left\{ \mathcal{NP}\cap\mathcal{C}_1\right\}=32.$ As an example of the technique used for the rest of cases, we will only present case $\mathcal{C}_9$. The reader interested in a full description of is referred to \cite{LiNi}.

\subsection{The class $\mathcal{C}_9$} To try to guess the maximum value in $\mathcal{NP}\cap \mathcal{C}_9 = \mathcal{NP}\cap \left\{90 + 11k, \, k\geq 0\right\}$, we take $a=9$, since $90 = 10\cdot a$. To get the necessary condition $k=b\geq 2a+1$, we must consider $b\geq 19$. 
Therefore, the first values of $\mathcal{C}_9$, for $b=0,1,\ldots,18$ are elements of $\mathcal{NP}$ given by $ n = 90 + 11b,$ with $b\in\{0,1,\ldots, 18\}$. So, we start our analysis with $b\geq 19$.

$\bullet$ If  $b \neq \dot{3}$, then directly $n=10\cdot 9 + 11\cdot b\in\mathrm{Per}(F_4)$ for all $b\geq 19$. 

$\bullet$ Assume that $b=3j$, with $j\geq 7$ and consider $n=10\cdot 9 + 11\cdot b = 10\cdot 9 + 11 \cdot (3j) = 10\cdot 31 + 11\cdot (3j-20).$

-- If, additionally, $3j-20$ is not a multiple of $31$, then $n$ is a period of $F_4$ if $3j-20\geq 63$ or $j\geq 28$ (notice that it must be $3j-20 \geq 2b'+1$ with $b'=31$). So, we must analyze the exceptions $7\leq j\leq 27$. For these values, we find that $321,$ $354,$ $387,$ $420,$ $453,$ $486,$ $519,$ $552,$ $585,$ $618,$ $684,$ $750,$ $816,$ $882,$ $915,$ $948$ belong to $\mathcal{NP}$, and on the other hand $\{651,$ $717,$ $783,$ $849,$ $981\}\subset \mathrm{Per}(F_4)$ with associate pairs $[a',b']$ given respectively by $[20,47],[20,53],[20,59],[20,71],$ $[31,64].$
 
-- If $3j-20=\dot{31}$, $j\geq 7$, then $j=17+31u,\, u\geq 0.$ In this point, we consider the general formulation of the decompositions in $\mathcal{C}_9$, 
  \begin{equation}\label{Eq:paraCaso9}
n=10\cdot (9+11r) +11\cdot (3j-10r), \, r\geq 0.
\end{equation} 

The sequence $\left\{x^{(9)}_r:r\geq 0\right\}=\left\{9+11r:r\geq 0\right\}=\left\{9,20,31,42,53,\ldots\right\}$ includes  prime numbers which are periods of $F_4$, for instance  $x^{(9)}_{8}=97$. In this way, if in the decompositions (\ref{Eq:paraCaso9}) we are not able to find a value $r=1,\ldots, 8$ such that $\gcd(9+11r,3j-10r)=1$, at least we know that $n$ can be divided by $3$, $31$,   as well as $97$, so $n=97\cdot t$, with $t\geq 93$. Then, Proposition~\ref{P:formacionP} gives $n\in\mathrm{Per}(F_4)$. For $r=8$ it is necessary that in (\ref{Eq:paraCaso9}) either $3j-10\cdot 8\geq 2\cdot (97)+1$ if $\gcd(97,3j-10\cdot 8)=1$ or $3j- 10\cdot 8\geq 1$, otherwise; in both cases, it suffices to take $j\geq 92$. Taking into account that $j=17+31u,$ $u\geq 0$, the inequality $j\geq 92$ is true when $u\geq 3$. In this case, the numbers $n$ will be in $\mathrm{Per}(F_4)$. For $u=0,1,2$, we have $n=651, 1674, 2697$, being $651$ and $2697$ in $\mathrm{Per}(F_4)$. Finally, $1674$ provides us the maximum of $\mathcal{NP}\cap\mathcal{C}_9$ (notice that $1674=10\cdot 9+11\cdot 144 = 10\cdot 20+11\cdot 134 =10\cdot 31+11\cdot 124 =10\cdot 42+11\cdot 114=10\cdot 53+11\cdot 104 =\cdots,$ and realize that $104<2\cdot 53 +1$. 
This concludes the part for a multiple of $3$, $b=3j$, and ends the casuistic. We have found that $$N_9:=\max\left\{\mathcal{NP}\cap \mathcal{C}_9\right\}= 1674.$$ 

\subsection{The final bound. A table of periods of $\mathrm{\mathbf{Per}}\mathbf{(F_4)}$} 
By collecting all the study developed for the sets $\mathcal{C}_m, 1\leq m\leq 10$ (see \cite{LiNi}), if we denote by $N_m$ the maximum value in 
$\mathcal{NP}\cap \mathcal{C}_m$, $1\leq m\leq 10$, we find:

\begin{tabular}{ccccc}
$N_1= 32,$ & $N_2=1560,$ & $N_3=1350,$ & $N_4=1140,$ &  $N_5=1260,$\\
$N_6=918,$ & $N_7=840,$ & $N_8=1026,$ & $N_9=1674,$ & $N_{10}=1332.$
\end{tabular}

On the other hand, the maximum value of $\mathcal{NP}$ being a multiple of $11$ is $N_{11}=1320$ as Proposition~\ref{P:para11} and Corollary~\ref{C:cota11} show. Therefore, $$M=\max\{\mathcal{NP}\}=\max\{N_m: 1\leq m \leq 11 \}=\boxed{1674}.$$ Now, with the help of a mathematical software and a few of patient we obtain all the periods in $\mathrm{Per}(F_4)$, which we gather in the following table:

\begin{center}
\resizebox{11cm}{!} {
\begin{tabular}{c|c}
\hline Intervals & Periods in $\mathrm{Per}(F_4)$ \\ \hline
 $n\in[1,100]$ & $1,8,11,43,54, 65, 75, 76, 87, 97, 98$ \\ \hline 
 $n\in[101,200]$ & $107, 109, 118, 119, 120, 131, 139, 140, 141, 142, 151, 153,$  \\ 
   $ $   & $161, 163, 164, 171, 173, 175, 182, 183, 184, 185, 186, 193, 197$  \\
 \hline
$n\in[201,300]$ &  $203, 204, 205, 206, 207, 208, 217, 219, 226, 227, 229, 230, 235, 237, 239,$\\
$ $ & $241, 246, 247, 248, 249, 250, 251, 252, 257, 259, 263, 267, 268, 269, 271,$ \\
$ $ & $ 272, 273, 274, 279, 281, 283, 285, 289, 290, 292, 293, 295, 296, 299$\\
\hline
$n\in[301,400]$ & $303, 305, 307, 311, 312, 313, 314, 315, 316, 317, 318, 323, 329, 331, 332, $\\
 $ $ & $333, 334, 335, 336, 337, 338, 339, 340, 343, 345, 347, 349, 351, 353, 355, $\\
 $ $ & $356, 358, 359, 361, 362, 363, 365, 367, 369, 371, 373, 374, 376, 377, 379, $\\
 $ $ & $381, 382, 383, 384, 385, 389, 391, 395, 396, 397, 398, 399, 400$\\
\hline 
$n\in[401,500]$ & $[401,500] \setminus \{408, 410, 412, 414, 416, 420, 423, 426, 430, $\\
 $ $ & $432, 434, 435, 436, 452, 453, 454, 455, 456, 458, 473, $\\
 $ $ & $474, 476, 478, 480, 485, 486, 490, 492, 496, 498, 500 \}$\\
\hline 
$n\in[501,600]$ & $[501,600] \setminus\{518,519,520,522,525,532,540,542,544,546,$ \\
 & $552,558,562,564,584,585,586,590,594,595,600 \}$ \\
\hline 
$n\in[601,700]$ & $[601,700]\setminus\{606,608,609,612,618,624,$ \\ & $628,650,672,678,684,686,690,700\}$\\
\hline 
$n\in[701,800]$ & $[701,800]\setminus\{702, 705, 710, 738, 744, 750, 756\}$\\
\hline 
$n\in[801,900]$ & $[801,900]\setminus\{804, 810, 820, 826, 836, 840, 870, 876, 882\}$\\
\hline 
$n\in[901,1000]$ & $[901,1000]\setminus\{915, 918, 920, 930, 936, 942, 948, 980, 988\}$\\
\hline 
$n\in[1001,1100]$ & $[1001,1100]\setminus\{1002, 1008, 1020, 1026, 1030, 1068, 1074\}$\\
\hline 
$n\in[1101,1300]$ & $[1101,1300]\setminus\{1134, 1140, 1200,1260, 1266, 1274, 1288\}$\\
\hline 
$n\in[1301,1700]$ & $[1301,1700]\setminus\{1320, 1332, 1350,1560, 1674\}$\\
\hline 
$n > 1674$ & All the values\\ \hline
\end{tabular}
}
\end{center}

\section{Further developments. Open questions} \label{S:Further}

Once we have studied in detail the periodic character of the solutions of Equation~(\ref{Eq:G4}), a problem that is still open is to describe $\mathrm{Per}(F_k)$ for $k\geq 5$. We believe that the techniques developed in the present paper could be helpful for a deeply analysis of the problem. As a first step, it would be instructive to look at the case $k=5$ and to try, at least, to prove if the set $\mathbb N \setminus \mathrm{Per}(F_k)$ is bounded. In this sense, we have realized computer simulations that indicate that the set of periods it is also determined by a linear combination of certain numbers. For instance, if we consider Equation (\ref{Eq:Gk}) for $k=5$, the period of every periodic sequence that we have generated is a linear combination of $13$ and $14$, that is, $N=13\cdot a + 14 \cdot b$ with $a, b \in \mathbb{N}$ verifying $\gcd(a,b)=1$. Moreover, for $k=6$, the periods that have appeared are of the form $N=16\cdot a + 17 \cdot b$ with $a, b \in \mathbb{N}$ verifying $\gcd(a,b)=1$. Realize that in both cases, as it happens for $k=4$, the numbers that determine the linear combinations are $3k-1$ and $3k-2$. Therefore, such simulations, jointly with Propositions \ref{Prop_equi} and \ref{Remark_Per2}, motivate us to state the following conjecture:
\begin{conjecture}
Consider Equation (\ref{Eq:Gk}) and let $\mathrm{Per}(F_k)$ be its set of periods. Then, $$\mathrm{Per}(F_k) \subset  \{1, \frac{1}{2} (3-(-1)^k), 2k, 3k-1\} \cup \{ (3k-2)\cdot a + (3k-1) \cdot b \ | \gcd(a,b) = 1 \}.$$ 
\end{conjecture}
Observe that we do not consider the equality since we hardly believe that an extra restriction as $b\geq 2a +1$ for $k=4$ must be considered. 

Moreover, a complete description of the elements of $\mathrm{Per}(F_k)$ is related to a coin problem with additional restrictions on the coefficients of the linear combinations, namely, $\gcd(a,b)= 1$ and certain linear inequalities on $a$ and $b$. As far as we know, this modified coin problem has not been previously studied in the literature. Therefore, some general interesting questions for algebrists would be: to study the boundedness character of the complementary of semigroups generated by two coprimes numbers, $p$ and $q$, whose elements are originated by combinations $a\cdot p+b\cdot q$ being $a, b$ natural integers satisfying $\gcd(a, b) = 1$, as well as to find in the bounded case a formula for the biggest positive integer not representable with the assumptions of being coprimes $a, b$; to extend the previous problem for $p, q$ not being necessarily coprime; to analyse the problem of adding extra conditions to $a, b$ apart from the fact of being coprime, as for example, linear inequalities.

Furthermore, the problem of studying the periodic structure of Equation~(\ref{Eq:Gk}) when we change the constant $0$ in the maximum function for a general positive number is open too. To this respect, in \cite{Barb} the cases $1,-1$ are considered for $k=2,3$.

Also, (\ref{Eq:Gk}) can be transformed into $z_{n+k} = \frac{\max\{z_{n+k-1}, z_{n+k-2}, \ldots, z_{n+1}, 1 \}}{z_n}$ under the change of variables $x_n = \ln (z_n)$. Related to this equation, Grove and Ladas proposed the following conjecture in \cite[Conjecture 2.2]{GroveL05}.

\begin{conjecture}
Assume $A \in (0,\infty)$. Show that no positive non-equilibrium solution of the equation $x_{n+1} = \frac{\max\{A, x_n, x_{n-1}\} }{x_{n-2}}$, has a limit. Extend and generalize.
\end{conjecture}

We believe that the oscillatory character of the solutions and the existence of an invariant function (consult \cite{Barb}) will be useful tools for extending and generalizing the conjecture.

\section*{Acknowledgements}

We sincerely thank Prof. Pedro A. Garc\'{i}a S\'{a}nchez from University of Granada, Spain, and Prof. Christopher O'Neill from San Diego Statal University, CA, USA, for useful conversations on the modified coin problem. 

This paper has been partially supported by Grant MTM2017-84079-P (AEI/FEDER,UE), Ministerio de Ciencias, Innovaci\'{o}n y Universidades, Spain.

\end{document}